\documentclass[hidelinks,onefignum,onetabnum]{siamart220329}

\usepackage{comment}
\usepackage{amsfonts}
\usepackage{amssymb}
\usepackage{graphicx}
\usepackage{epstopdf}
\usepackage{algorithmic}
\usepackage{enumerate}
\usepackage{enumitem}
\usepackage{mathrsfs}
\usepackage{mathtools}
\usepackage{subcaption}
\allowdisplaybreaks

\newcommand{\di}[1]{{\mathrm{div}}\left( #1\right)}
\newcommand{\ve}{\varepsilon}
\newcommand{\A}{\mathcal A}

\definecolor{ocre}{RGB}{243,102,25}

\ifpdf
  \DeclareGraphicsExtensions{.eps,.pdf,.png,.jpg}
\else
  \DeclareGraphicsExtensions{.eps}
\fi

\newsiamremark{remark}{Remark}
\newsiamremark{hypothesis}{Hypothesis}
\crefname{hypothesis}{Hypothesis}{Hypotheses}
\newsiamthm{claim}{Claim}

\headers{Multidomain models for nerve bundles with random structure}{I. Pettersson, A. Rybalko, and V. Rybalko}

\title{Nonlinear multidomain model for nerve bundles with random structure\thanks{
\funding{This work was supported by Swedish Foundation for International Cooperation in Research and Higher education (STINT) project CS2018-7908, Olle Engkvist Foundation Grant 227-0235, and SSF grants for Ukrainian scientists Dnr UKR22-0004, UKR24-0023.}}}

\author{Irina Pettersson\thanks{Chalmers University of Technology and University of Gothenburg, Sweden 
  (\email{irinap@chalmers.se}%, \url{https://research.chalmers.se/person/irinap}
).}
\and Antonina Rybalko\thanks{Chalmers University of Technology and University of Gothenburg, Sweden; Kharkiv National University of Radio Electronics, Ukraine (\email{rybalkoa@chalmers.se}).}
\and Volodymyr Rybalko\thanks{Chalmers University of Technology and University of Gothenburg, Sweden; Institute for Low Temperature Physics and Engineering, Ukraine (\email{rybalko@chalmers.se}).}}

\usepackage{amsopn}

\begin{document}

\maketitle

% REQUIRED
\begin{abstract}
We present a derivation of a multidomain model for the electric potential in bundles of randomly distributed axons with different radii. The FitzHugh-Nagumo dynamics is assumed on the axons' membrane, and the conductivity depends nonlinearly on the electric field. Under ergodicity conditions, we study the asymptotic behavior of the potential in the bundle when the number of the axons in the bundle is sufficiently large and derive a macroscopic multidomain model describing the electrical activity of the bundle. Due to the randomness of geometry, the effective intracellular potential is not deterministic but is shown to be a stationary function with realizations that are constant on axons' cross sections. The technique combines the stochastic two-scale convergence and the method of monotone operators.
\end{abstract}

% REQUIRED
\begin{keywords}
Signals propagation in nerves, nonlinear conductivity, random media, stochastic homogenization, bidomain model.
\end{keywords}

% REQUIRED
\begin{MSCcodes}
Primary: 35B27, 78A48, 35K61, 35K57; Secondary: 35K65, 92C30.
\end{MSCcodes}

\section{Introduction}	
The axons in the nervous system are grouped into fascicles, and the number of axons and their diameters vary from one fascicle to another, as well as inside one fascicle \cite{standring2021gray}.  
The signal propagation along individual axons is traditionally modeled by nonlinear cable equations assuming the Hodgkin-Huxley dynamics \cite{hodgkin1952} or its simplifications, such as FitzHugh--Nagumo model \cite{fitzhugh1955mathematical, nagumo1962active}. In \cite{jerez2020derivation} a nonlinear cable equation for signal propagation in a single myelinated axon has been derived from a three-dimensional model.  
	
To catch nontrivial interactions between the axons in the bundle (ephaptic interactions \cite{bokil2001ephaptic, binczak2001ephaptic}), one needs to consider a full three-dimensional model of the bundle with several (and quite often a large number) axons in it. Complicated geometry and nonlinear dynamics on the membrane make it, however, computationally expensive, and so macroscopic models are used to describe the electric potential in the axon bundles. It was suggested in \cite{BAR00} that one can model both peripheral nerves, cortical neurons, and syncytial tissues with the help of bidomain models. In \cite{mandonnet2011role} it was hypothesized that the homogenization procedure for a bundle of axons leads to a bidomain model. In \cite{JEREZHANCKES2023103789} the authors have confirmed this hypothesis and derived an effective bidomain model for the nonlinear case, assuming FitzHugh dynamics on the membrane. In the latter work, the myelinated axons of the same radius are placed periodically inside the fascicle. 
 
 % The characteristic microscale of the structure is given by the small positive parameter $\ve$. The distance between the individual axons in the bundle, the spacing of the Ranvier nodes (unmyelinated parts of the membrane), and the diameters of axons are assume to be of order $\ve$.  

In the above-described models, the periodicity of the axon distributions inside a fascicle or of the geometry of syncytial tissues is a crucial assumption allowing to apply the classical homogenization techniques. In reality, however, the placement of the axons is not periodic, and the diameters of individual fibers in one fascicle may vary. Since the diameters of axons affect the conduction velocity, it is important to allow for such variation. 
	
There have been several attempts to remove the periodicity assumption. 
%In \cite{keener1996biophysical}, the authors propose a model for the electrical activity of cardiac tissue that takes into account cellular microstructure. This model allows small-scale variations in resistance and larger-scale tissue variations, such as rotational anisotropy and fiber curvature. Numerical experiments suggest a possibility of direct activation and defibrillation not seen in previous models. The effect of the microstructure of the medium is incorporated through the ionic current. 
	% It is shown that the only transformation satisfying the assumptions they made is rotation, so I do not want to cite this work.
% 	
For example, in \cite{richardson2011}, a bidomain model for a beating heart with a non-periodic microstructure is proposed. Namely, the authors apply formal multi-scale expansions to problems stated in deformed periodic structures. The model accounts for the nonuniform orientation of the cells and the deformation of the tissue as a result of the heartbeat. In \cite{pettersson2023bidomain}, the asymptotics of the electric potential in bundles with randomly distributed identical axons have been studied. It has been shown that the macroscopic behavior of the potential is given by an effective deterministic bidomain model. 
	
The present paper generalizes the homogenization results \cite{JEREZHANCKES2023103789} and \cite{pettersson2023bidomain} to the case of nonlinear conductivity and random geometry of the nerve cross section. Namely, we place the axons randomly in the cross section of the bundle, keeping the cross section of the axons the same along the bundle. Typically, the randomness in the cross section is generated by a stationary and ergodic point process. The evolution of the membrane potential, that is, the jump of the electric potential through the membrane, is modeled with the help of the FitzHugh--Nagumo dynamics. For a large number of axons in a bundle, we derive an effective model \eqref{eq:hom-prob} of the multidomain type. {The key contribution of the present work is the derivation of a non-deterministic  effective problem. The randomness of the geometry is encoded in the function $r(\omega)$ (see \eqref{r(omega)}) and the Palm measure \eqref{Palmm} corresponding to the surface measure of the axons' membrane. Considering stochastic geometry opens an opportunity to analyze, for example, how the concentration of axons of different radii influences the signal propagation in a bundle, which is not possible in a periodic setting.} 

To tackle the random geometry, one can make use of the ergodicity theorem directly \cite{zhikov1993averaging} or apply stochastic two-scale convergence introduced in \cite{bourgeat1994stochastic}. In \cite{blanc2007stochastic}, the authors suggest a different approach, defining the stochastic coefficients as stochastic deformations of the periodic setting using random diffeomorphisms. In the present work we use the stochastic two-scale convergence technique combined with the monotonicity method \cite{minty1962}. Another technical aspect of the homogenization procedure in our problem comes from the presence of the jump of the electric potential through the membrane. To pass to the limit in the surface terms, we use the Palm measure theory. A similar approach has been employed in \cite{piatnitski2020homogenization} for the proof of convergence of the boundary integrals involving the flow velocity and elastic deformations and also in \cite{heida2022stochastic-II} to tackle the Robin boundary conditions.

It is known that if the conductivity of the medium decreases, it is necessary to deliver a pulse of a larger field amplitude to achieve the same effect \cite{ivorra2010electrical}. The choice of a specific nonlinear conductivity is still phenomenological since the experimental data currently does not allow us to characterize the dependence of the conductivity on the electric field. In the present work, we impose hypothesis (H4) ensuring the monotonicity of the problem, which is clearly satisfied by, for example, the sigmoid-type functions.   
 
Bidomain models are traditionally used to describe the macroscopic behavior of the electric potential in cardiac electrophysiology  \cite{franzone2002degenerate, pennacchio2005, amar2013hierarchy, collin2018mathematical, GraKar2019, BenMroSaaTal2019}. In these works, the geometry of the tissue is periodic, and the conductivity is independent of the electric field. Moreover, both the  intra- and extracellular components of the cardiac tissue are connected, which stands in contrast to the case when axons are disconnected in the transversal direction. As a consequence, the effective multidomain model \eqref{eq:hom-prob} contains only derivatives of the intracellular potential in the longitudinal direction. 

The rest of this paper is organized as follows. In Section \ref{sec:geometry}, we define the random geometry of the nerve bundle. In Section \ref{sec:2scale}, we introduce the necessary definitions of the stochastic two-scale convergence in the mean, as well as the Palm measure. In Section \ref{sec:setup}, we formulate the microscopic problem, and Section \ref{sec:result} introduces the effective multidomain model. The proof of the main convergence result (Theorem \ref{th:main-short}) is given in Section \ref{sec:proof}. In the appendix, we provide some examples of stationary random media and perform the nondimensionalization of the problem to motivate the choice of the scaling in the microscopic model (the factor $\ve$ in \eqref{eq:orig-prob}).

\section{Geometry of nerve fascicles}
\label{sec:geometry}
A nerve fascicle is modeled as a cylindrical domain $G=(0,L)\times S$ in $\mathbb{R}^3$,  with length $L>0$ and cross section $S \subset \mathbb{R}^2$. We assume that $S$ has a Lipschitz boundary $\partial S$. The lateral boundary of the cylinder is denoted by $\Sigma=(0,L)\times \partial S$, and its bases by $S_0=\{0\} \times S$ and $S_L=\{L\} \times S$. In what follows, the points in $\mathbb R^3$ are denoted by $x = (x_1, x')$, $x_1\in \mathbb R, x'\in \mathbb R^2$.
	
We assume that the bulk of the fascicle consists of two disjoint separated by semipermeable membrane parts: an intracellular part formed by thin cylinders (axons) and an extracellular part. Let us describe the random geometry of the bundle that will originate from the random distribution of the axons in the cross section of the fascicle. 

Let $(\Omega,{\mathcal F}, {\mathcal P})$ be a probability space with measure ${\mathcal P}$ and the sigma-algebra ${\mathcal F}$. We will also assume that the Lebesgue space $L^p(\Omega, {\mathcal P})=L^p(\Omega, {\mathcal F}, {\mathcal P})$, $1\le p<\infty$, is separable. It is true if ${\mathcal F}$ is countably generated (see e.g. Proposition 3.4.5, \cite{cohn2013measure}).
	
Consider a two-dimensional dynamical system (of translations)  $T_{x'}:(\Omega,{\mathcal F}, {\mathcal P}) \rightarrow (\Omega,{\mathcal F}, {\mathcal P})$, $x'\in \mathbb R^2$, that is a family of invertible measurable maps with measurable inverses such that
\begin{itemize}
		\item[(i)]
		$T_0 = \rm Id$ and $ T_{x'+y'} = T_{x'} T_{y'}$, for all  $x', y' \in \mathbb{R}^2$.
		\item[(ii)]
		$T_{x'}$ is a measure-preserving transformation, that is ${\mathcal P} (T^{-1}_{x'}F) = {\mathcal P}(F)$ for all $x'\in \mathbb{R}^2, F\in {\mathcal F}$.
		\item[(iii)]
		$T_{x'}\omega$ is a measurable map from $(\mathbb{R}^2\times 
		\Omega, \mathcal B \times {\mathcal F}, dx' \times {\mathcal P})$ to $(\Omega, {\mathcal F})$, where $dx'$ is the standard Lebesgue measure and $\mathcal B$ is the Borel $\sigma$-algebra on $\mathbb{R}^2$.
\end{itemize}
We assume that $T_{x'}$ is ergodic, that is if $F\in {\mathcal F}$ is translation invariant, then $F$ has either full or zero measure:
$T_{x'} F = F, \,\, \forall x'\in \mathbb R^2 \Rightarrow {\mathcal P}(F)=0 \,\,\mbox{or}\, 1.$
	
Given $\omega\in \Omega$, consider the random set $A(\omega)$ defined as a collection of an infinite number of disjoint {open} discs with radii $r_j\in \mathbb{R}_+$:
\begin{align*}
&A(\omega)=\bigcup_{j=1}^{\infty}A_j(\omega), \\
&A_j(\omega)=A(x'_{j}(\omega),r_{j}(\omega))=\biggl\{x'\in \mathbb{R}^2: \left| x'-x'_{j}(\omega) \right| < r_{j}(\omega) \biggr\}.
\end{align*}
We assume that
\begin{enumerate}[label=(H\arabic*)]
\setcounter{enumi}{0}
\item \label{H1}
The distance between the discs is bounded from below:
\begin{align*}
	\mbox{dist} \bigl(A_{i}(\omega), A_{j}(\omega)\bigr)\geq d_0>0, \ \  \forall i\not= j.
\end{align*}
\item \label{H2} There exists $R_0>0$ such that a.s. any disc of radius $R_0$ in
 $\mathbb{R}^2$ has a nontrivial intersection with $A(\omega)$.
\item \label{H3}
The radii $r_j(\omega)$ of the discs, $\forall j$, are bounded from below and above by deterministic positive constants: 
$
0<\underline{r}<r_j(\omega)<\overline{r}
$.
\item \label{H4}
The dynamical system $T_{x'}$ is consistent with the random set $A(\omega)$, that is
\begin{align*}
A(T_{x'}\omega)= \bigcup_{j=1}^{\infty} A(x'_{j}(\omega)-x', r_{j}(\omega))  \ \ \mbox{for} \ {\mathcal P}-\mbox{a.s.} \, \mbox{in}\,\, \Omega.
\end{align*}
\end{enumerate}
We introduce prototypes of the random set $A(\omega)$, its closure $\bar A(\omega)$, and the boundary $\partial A(\omega)$:
\begin{align} 
&
\A =\{\omega\in \Omega: 0\in A(\omega)\}, \nonumber\\
%0<{\mathcal P}(\A )<1, 
& \bar{\A }=\{\omega\in \Omega: 0\in \bar{A}(\omega)\}, \label{def:prototypes}\\ 
&{\mathcal M}=\bar{\A }\setminus \A  = \{\omega\in \Omega: 0\in \partial A(\omega)\}.\nonumber
\end{align}	
Hypothesis \ref{H4} implies that
\begin{align}
	\label{def:cal-A}
    A(\omega)=\lbrace{ x'\in \mathbb{R}^2: T_{x'}\omega \in \A  \rbrace}.
\end{align}
Indeed, if $T_{x'}\omega \in \A $ then  $ 0\in \bigcup_{j=1}^{\infty} A(x'_{j}(\omega)-x', r_{j}(\omega))$, and consequently $ x'\in \bigcup_{j=1}^{\infty} A(x'_{j}(\omega), r_{j}(\omega))=A(\omega).$ Similarly, $\partial A(\omega)=\lbrace{ x'\in \mathbb{R}^2: T_{x'}\omega \in {\mathcal M} \rbrace}$.
 
For a small $\ve >0$ representing the microscale, we rescale the set $A(\omega)$: 
\begin{align}
\label{eq:A_eps}
A_\ve(\omega)= \bigcup_{j \in J_\ve} A_{\ve j}(\omega), \quad A_{\ve j}(\omega)=\ve A(x'_{j}(\omega),r_{j}(\omega)),
\end{align}
where the union is taken over the indices $j\in J_\ve$ such that the discs $A_{\ve j}(\omega)$ are separated from the boundary of $S$, namely $\mbox{dist} \bigl(A_{\ve j}(\omega), \partial S\bigr) > d_0\ve$, $d_0>0$. 

The random radii of the discs in the cross section is given by the following function:
\begin{align*}
	r_\ve(x,\omega) =
	\begin{cases}
		r_j(\omega),  \quad x\in \bar A_{\ve j}(\omega),\\
		0,  \quad \hfill x\notin \bar A_{\ve j}(\omega).
	\end{cases}
\end{align*}
Then $r_\ve(x,\omega)=r(T_{x'/\ve}\omega)$, where
\begin{align}
\label{r(omega)}
	r(\omega) =
	\begin{cases}
		r_{j_0}(\omega),  \quad \omega\in {\bar \A },\\
		0,  \quad \hfill \omega\notin {\bar \A },
	\end{cases}
\end{align}
and $r_{j_0}(\omega)>0$ is the radius of the disc $A(x'_{j_0}(\omega),r_{j_0}(\omega))$ containing 0 (if exists). 

The intracellular domain of $j$th axon is denoted by
$
F_{\ve j}(\omega)=(0,L) \times A_{\ve j}(\omega)$.
{The intracellular domain is then defined  as a random set of axons, and the extracellular is the complement to the intracellular one:
\begin{align}
\label{eq:F_eps}
	F_\ve(\omega)=(0,L)\times A_\ve(\omega), \quad
    G_\ve(\omega)=G\setminus F_\ve(\omega).
	%\lbrace{ x\in \mathbb{R}^2: T_{x/\ve}\omega \in \mathcal F \rbrace\cap G;}
\end{align}
}
% The extracellular domain is the complement to the intracellular one:
% \begin{align}
% \label{eq:G_eps}
% 	G_\ve(\omega)=G\setminus F_\ve(\omega).
% \end{align}
	%\begin{comment}
The membrane (lateral boundary) of axons is denoted by 
\begin{align}
\label{eq:M_eps}
M_\ve(\omega)=\partial F_\ve(\omega)\setminus (S_0\cup S_1).
\end{align}
{The geometry of the cross section and the three-dimensional bundle is illustrated in Figure \ref{fig:cross-section}.}
\begin{figure}[H]
\centering 
\def\svgwidth{.8\textwidth}
%% Creator: Inkscape 1.3.2 (091e20e, 2023-11-25), www.inkscape.org
%% Accompanies image file 'random-bundle.pdf' (pdf, eps, ps)
\begingroup%
  \makeatletter%
  \providecommand\color[2][]{%
    \errmessage{(Inkscape) Color is used for the text in Inkscape, but the package 'color.sty' is not loaded}%
    \renewcommand\color[2][]{}%
  }%
  \providecommand\transparent[1]{%
    \errmessage{(Inkscape) Transparency is used (non-zero) for the text in Inkscape, but the package 'transparent.sty' is not loaded}%
    \renewcommand\transparent[1]{}%
  }%
  \providecommand\rotatebox[2]{#2}%
  \newcommand*\fsize{\dimexpr\f@size pt\relax}%
  \newcommand*\lineheight[1]{\fontsize{\fsize}{#1\fsize}\selectfont}%
  \ifx\svgwidth\undefined%
    \setlength{\unitlength}{1264.2519685bp}%
    \ifx\svgscale\undefined%
      \relax%
    \else%
      \setlength{\unitlength}{\unitlength * \real{\svgscale}}%
    \fi%
  \else%
    \setlength{\unitlength}{\svgwidth}%
  \fi%
  \global\let\svgwidth\undefined%
  \global\let\svgscale\undefined%
  \makeatother%
  \begin{picture}(1,0.38605831)%
    \lineheight{1}%
    \setlength\tabcolsep{0pt}%
    \put(0,0){\includegraphics[width=\unitlength,page=1]{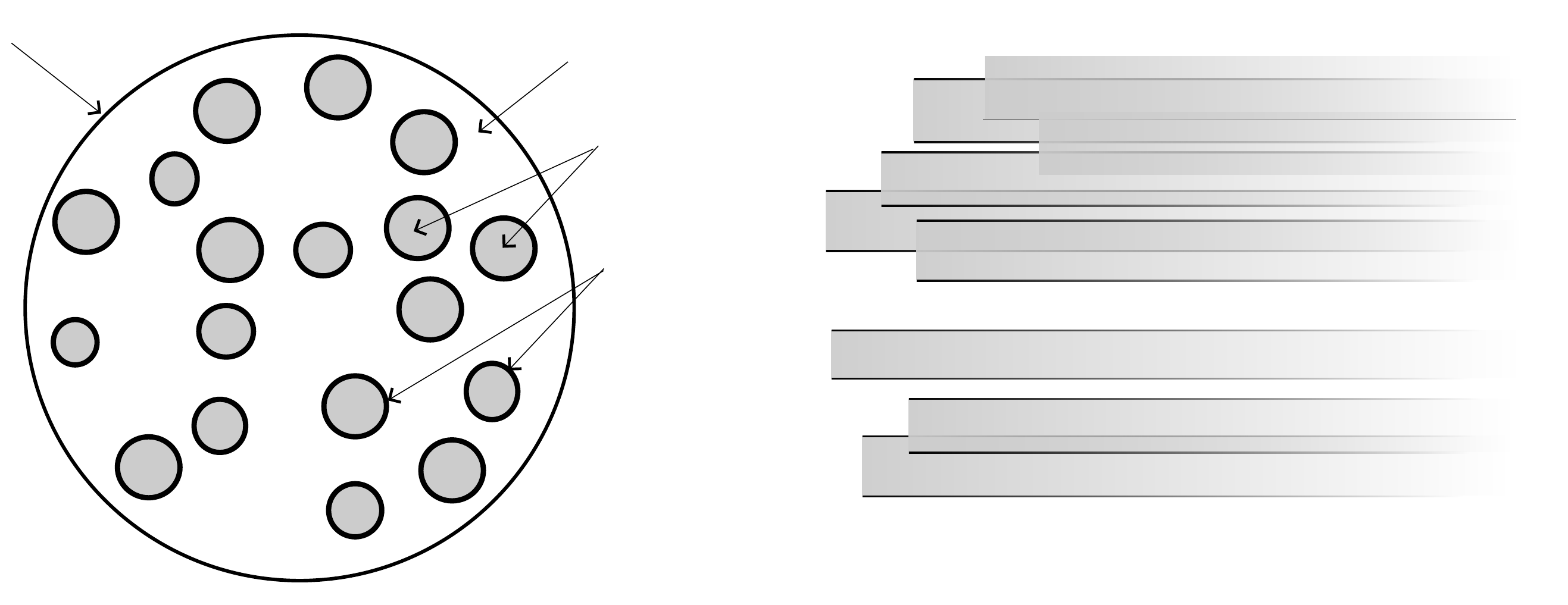}}%
    \put(0.00811085,0.36561019){\makebox(0,0)[lt]{\lineheight{1.25}\smash{\begin{tabular}[t]{l}$\Sigma$\end{tabular}}}}%
    \put(0.31073835,0.3547719){\makebox(0,0)[lt]{\lineheight{1.25}\smash{\begin{tabular}[t]{l}$G_\varepsilon(\omega)$\end{tabular}}}}%
    \put(0.35288063,0.30543295){\makebox(0,0)[lt]{\lineheight{1.25}\smash{\begin{tabular}[t]{l}$F_\varepsilon(\omega)$\end{tabular}}}}%
    \put(0.38340426,0.22600079){\makebox(0,0)[lt]{\lineheight{1.25}\smash{\begin{tabular}[t]{l}$M_\varepsilon(\omega)$\end{tabular}}}}%
    \put(0,0){\includegraphics[width=\unitlength,page=2]{random-bundle.pdf}}%
  \end{picture}%
\endgroup%

\caption{Fascicle $G$ with the lateral boundary $\Sigma$.}
\label{fig:cross-section}
\end{figure}

In what follows, we will use the Birkhoff ergodic theorem \cite{cornfeld2012ergodic, shiryaev2016probability}.
	\begin{theorem}[Birkhoff Ergodic Theorem]
		Let $f\in L^p(\Omega)$, $p\ge 1$. Then for almost all $\omega \in \Omega$, the realization $f(T_x\omega)$ possesses a mean value, that is
		\begin{align*}
			f(T_{x/\ve} \omega) \rightharpoonup \langle f(T_x \omega)\rangle = \lim_{\ve \to 0} \frac{1}{|K|}\int_K f(T_{x/\ve} \omega)\, dx.
		\end{align*}
		Moreover, the mean value $\langle f(T_x \omega)\rangle$ considered as a function of $\omega \in \Omega$, is invariant, and
		\begin{align*}
			\int_\Omega \langle f(T_x \omega)\rangle d{\mathcal P}(\omega)  = \int_\Omega f(\omega) \, d{\mathcal P}(\omega).
		\end{align*}
		If the system $T_x$ is ergodic, then  
		\begin{align*}
			\langle f(T_x \omega)\rangle = \int_\Omega f(\omega)\, d{\mathcal P}(\omega).
		\end{align*}
\end{theorem}
Applying the Birkhoff ergodic theorem, we conclude that ${\mathcal P}$-a.s. there exists a limit density of axons as $\ve \to 0$:
\begin{align}
	\chi_{_{A_\ve}}\rightharpoonup \Lambda={\cal P}({\cal A})\ \ \mbox{weakly in}\ \ L^p(G), \,\,p\ge1.
	\label{dens}
\end{align} 
	%\end{comment} 

%%%%%%%%%%%%%%%%%%
\section{Stochastic two-scale convergence in the mean} \label{sec:2scale}
In this section we recall the necessary definition of the two-scale convergence in the mean \cite{bourgeat1994stochastic} and introduce the notion of the two-scale convergence on the random surfaces. We also formulate the hypotheses \ref{H6}--\ref{H8} on the initial data $V_\ve, G_\ve$ and the external excitation $J_\ve^e$. 

Consider the group $\{U(x'): \, x'\in \mathbb R^2\}$ of operators on $L^2(\Omega, {\mathcal P})$ given for $f\in L^2(\Omega, {\mathcal P})$ by
\begin{align*}
	(U(x')f)(\omega) = f(T_{x'} \omega), \quad x'\in \mathbb R^2, \,\, \omega \in \Omega.
\end{align*}
$U(x')$ preserves inner product $(U(x')f, U(x')g)_{L^2(\Omega, {\mathcal P})} = (f,g)_{L^2(\Omega, {\mathcal P})}$ since (ii) is equivalent to 
\begin{align*}
	\int_\Omega f(\omega)\, d{\mathcal P} = \int_\Omega f(T_{x'} \omega)\, d{\mathcal P}, \quad \mbox{for all}\,\, f\in L^1(\Omega, {\mathcal P}).
\end{align*}
%due to the density of step functions in $L^1(\Omega, {\mathcal P})$. 
As $T_x$ is invertible, $U(x')$ is a unitary operator \cite{walters2007ergodic}. 
%The operator $U(x')$ is also called the Koopman operator of $T_{x'}$.

{Under the assumption that $L^2(\Omega, {\mathcal P})$ is separable, the group $\{U(x'): \,\, x'\in \mathbb R^2\}$ is strongly continuous, that is $U(x') f \to f$ strongly in $L^2(\Omega, {\mathcal P})$, as $x'\to 0$ (see Theorems 3.5.3 and 10.10.1 in \cite{hille1996functional}).} 

Denote by $(\nabla_\omega \,\cdot)_j$, $j=2, 3$, the generators of the strongly continuous group of unitary operators in $L^2(\Omega, \mathcal P)$ associated with translations $T_{x'}$ along $e_j$, the $j$th basis vector in $\mathbb R^3$. The domains $\mathcal D_j$ of these generators, $j=2, 3$, are dense in $L^2(\Omega, {\mathcal P})$ and
\begin{align*}
    (\nabla_\omega u)_j(\omega) = \lim_{\delta \to 0} \frac{u(T_{\delta e_j} \omega) - u(\omega)}{\delta}.
\end{align*}
% Note that we do not need to define stochastic derivatives for $j=1$ since the randomness is defined in the cross-section and is only with respect to the transverse variables $x_2, x_3$.
{For each multi-index $\alpha=(\alpha_1,\alpha_2)$, we set $D^\alpha = (\nabla_\omega u)_1^{\alpha_1}(\nabla_\omega u)_2^
{\alpha_2}$ and define
\begin{align}
{\cal D}(\Omega)&={\cal D}_2\cap {\cal D}_3,\nonumber\\
\label{def:D^inft}
D^\infty(\Omega)&=\{f\in L^\infty(\Omega, \mathcal P): D^\alpha f\in L^\infty(\Omega, \mathcal P)\cap {\cal D}(\Omega), \, \forall \alpha\}.
\end{align}
}
Note that $D^\infty(\Omega)$ is dense in $L^2(\Omega, \mathcal P)$ (see Lemma 2.1 in \cite{bourgeat1994stochastic}). 
To formulate the assumptions and the main result we need the notion of stochastic two-scale convergence in the mean \cite{bourgeat1994stochastic} both in $G_T=[0,T]\times G$ and on $[0,T] \times M_\ve(\omega)$.
\begin{comment}
\begin{definition}\label{def:admissible-G} A function $\varphi(t, x,\omega) \in L^2(G_T; L^2(\Omega, {\mathcal P}))$ is said to be admissible if $(t, x, \omega) \mapsto \varphi \left(t, x,T_{x'/\ve}\omega\right)$ is well-defined (measurable).
\end{definition} 
Note that linear combinations of products $\varphi(t, x) \psi(\omega)$ with $\varphi\in L^2(G_T)$, $\psi \in L^2(\Omega, {\mathcal P})$ are admissible (see Theorem III.11.17, \cite{dunford1988linear}).    
\end{comment}  
\begin{definition} 
\label{def:2scale-G}
A bounded sequence $\{u_\ve\}$ of functions in $L^2(G_T; L^2(\Omega, {\mathcal P}))$ converges  stochastically two-scale in the mean to
$u\in L^2(G_T; L^2(\Omega, {\mathcal P}))$, and we write $u_\ve \xrightharpoonup[]{2s} u$,
if for any %admissible (Definition \ref{def:admissible-G}) 
$\varphi \in C(\bar G_T; L^2(\Omega, {\mathcal P}))$, we have
$$
\lim_{\ve\to 0} \int\limits_{G_T} \int \limits_\Omega u_\ve(t, x,\omega) \varphi\left(t, x,T_{x'/\ve}\omega\right) \, dt dx d{\mathcal P}(\omega) =
\int\limits_{G_T} \int\limits_\Omega u(t, x,\omega)\varphi\left(t, x,\omega\right) \, dt dx d{\mathcal P}(\omega).
$$
\end{definition}
% \ip{Denote the surface measure of $\partial A(\omega)$ in $\mathbb R^2$ by $d\mu_\omega(x')=\chi_{\partial A(\omega)}(x') ds(x')$, where $ds(x')$ is the arc length measure and $\chi$ is the characteristic function. Then the surface measure of $\partial A_\ve(\omega)$ is 
% \begin{align}
% \label{def:measure}
% d\mu_\omega^\ve=\ve \chi_{S_\delta}(x') d\mu_\omega(x'/\ve)= \ve \chi_{\partial A_\ve(\omega)}(x')\, ds(x'),
% \end{align}
% where $S_\delta$ is defined in assumption (H5).
% }

To define the convergence on the random surfaces $M_\ve(\omega)$, we use the Palm measure corresponding to the surface measure of $M_\ve$ \cite{mecke1967stationare, miyazawa1995note}. 
\begin{lemma}[Palm measure $\mu$ and Campbell's formula]\label{lm:Campbell} There exists a unique measure $\mu$, called the Palm measure, on $\Omega$ such that
\begin{align}
\mu (B) =\int\limits_{\Omega} \int\limits_{\partial A(\omega)} \chi_{[0,1]^2}(x')\chi_{B}(T_{x'}\omega) ds(x')  d{\cal P}(\omega),
\label{Palmm}
\end{align}
% \ip{or maybe
% \begin{align}
% \mu (B) =\int\limits_{\Omega} \int\limits_{[0,1]^2} \chi_{B}(T_{x'}\omega) d\mu_\omega(x')\, d{\cal P}(\omega),
% \label{Palmm}
% \end{align}
% }
where $ds(x')$ is the arc length measure of $\partial A(\omega)$, and $\chi$ is the characteristic function. Moreover, $\mu$ is $\sigma$-finite and
\begin{align}
\int_\Omega \int_{\partial A(\omega)} f(x', T_{x'}\omega)\, ds(x') \,d{\mathcal P}(\omega)
= 
\int_\Omega \int_{\mathbf R^2} f(x', \omega)\, dx' d\mu(\omega),
\label{Campbell}
\end{align}
for any $f(x', \omega) \in L^1(\mathbf R^2 \times \Omega, dx'\times d\mu)$.
\end{lemma}
% The Palm measure $\mu$ can be interpreted as a pushforward measure of the measure $f(x')ds(x')d{\mathcal P}(\omega)$ on $\partial A(\omega)$ under the mapping $(x',\omega) \to T_{x'}\omega$. 
The ergodicity implies that $\mu$ is independent of the choice of $f$.\\
\begin{remark}
Taking $f(x',\omega)$ such that $f=0$ for $\omega\in \mathcal M$, and using the definition of the prototype set $\mathcal M$, we see that for all such functions
\begin{align*}
\int\limits_\Omega \int \limits_{\mathbb R^2} f(x',\omega)\, dx' d\mu(\omega)=0.
\end{align*}
Thus, the integral over $\Omega$ in the right-hand side of\eqref{Campbell} can be replaced by the integral over $\mathcal M$. The same applies also to the two-scale convergence definitions below.
\end{remark}

% A version of the Birkhoff ergodic theorem for the Palm measure $\mu$ can be proved as Lemma 2.14 \cite{heida2017stochastic}. Namely, for $\varphi\in C(\overline G)$, $\psi \in L^1(\Omega; \mu)$, and a typical realization $\omega$, we have
% \begin{align}
% \lim_{\ve \to 0} &\ve \int\limits_0^T \int\limits_0^L \int \limits_{\partial A_\ve(\omega)} \varphi(t, x_1, x') \psi(T_{x'/\ve}\omega) dt dx_1 ds(x')\\
% &= \int\limits_\Omega \int\limits_{G_T} \varphi(t, x_1, x') \psi(\tilde \omega)\, dt dx d\mu(\tilde \omega).
% \end{align}
% \todo{I am not sure we need this, but I let it stay here for now.}
%\ip{Following \cite{wright1994stochastic}, we introduce admissible test functions.}
\begin{comment}
    \begin{definition}[Admissible functions on random surfaces]\label{def:admissible-surface} 
A function $\varphi(t, x,\omega) \in L^q(G_T; L^q(\Omega, \mu))$, $1<q<\infty$, is said to be admissible if $(t, x, \omega) \mapsto \varphi \left(t, x,T_{x'/\ve}\omega\right)$ is measurable ... 
\end{definition} 
For example, $\varphi \in C(\bar G_T; L^q(\Omega, \mu))$ is admissible.
\end{comment}

\begin{definition}[Stochastic two-scale convergence on random surfaces] \label{def_2s_on_ax_t} A sequence $\{u_\ve=u_\ve(t,x,\omega)\}$ of functions on the random surface $M_\ve$, such that for some $1<p<\infty$,
% \begin{align*}
% \ve \int \limits_\Omega \int\limits_0^T\int\limits_{0}^{L} 
%    	\int\limits_{\partial A_\ve(\omega)}
% |u_\ve|^p\, ds(x')\, dx_1\, dt\, d{\mathcal P}(\omega) \le C,
% \end{align*}
\begin{align*}
\int \limits_\Omega \int\limits_0^T\int\limits_{M_\ve(\omega)} 
|u_\ve|^p \, dS(x_1, x')\, dt\, d{\mathcal P}(\omega) \le C,
\end{align*}
converges  stochastically two-scale in the mean to
$u=u(t,x,\omega)\in L^p(G_T; L^p(\Omega, \mu))$ if
\begin{align*}
   	\lim_{\ve\to 0} &\,\, \ve \int\limits_{\Omega} \int \limits_0^T  
   	\int\limits_{M_\ve(\omega)} u_\ve\, \varphi\left(t, x,T_{x'/\ve}\omega\right) \,dS(x)\, dt\,  d{\mathcal P}(\omega)=
   	\int\limits_{\Omega}  \int\limits_{G_T} u\,\varphi\left(t, x,\omega\right)\, dx\,dt\,  d\mu(\omega),
\end{align*}
for any %admissible 
$\varphi \in C(\bar G_T; L^q(\Omega, \mu))$, $\displaystyle q= p/(p-1)$. Here we write $dS(x_1, x')= dx_1 ds(x')$ for the surface measure on $M_\ve(\omega)$.
\end{definition}

We will also use stochastic two-scale convergence in the mean pointwise in time $t\in [0,T]$. 
\begin{definition} 
\label{def_2s_on_ax}We say that s sequence $\{u_\ve=u_\ve(t,x,\omega)\}$ of functions on the random surface $[0,T]\times M_\ve(\omega)$ such that for $1<p<\infty$,
\begin{align*}
    \ve \int \limits_\Omega  
   	\int\limits_{M_\ve(\omega)}
|u_\ve|^p\, dS(x)\, d{\mathcal P}(\omega) \le C,
\end{align*}
converges stochastically two-scale in the mean to
$u=u(t,x,\omega)\in L^p(G_T; L^p(\Omega, \mu))$, pointwise in time, if for any $t\in [0,T]$
%for any  $\varphi \in C(G\times\Omega)$ we have
\begin{align*}
   	\lim_{\ve\to 0} &\ \ve \int\limits_{\Omega}  
   	\int\limits_{M_\ve(\omega)} u_\ve\, \varphi\left(t, x,T_{x'/\ve}\omega\right) dS(x) d{\mathcal P}(\omega)=
   	\int\limits_{\Omega}  \int\limits_{G} u\,\varphi\left(t, x,\omega\right) dx  d\mu(\omega),
\end{align*}
for any %admissible (Definition \ref{def:admissible-surface})
$\varphi \in C(\bar G_T; L^q(\Omega, \mu))$, $\displaystyle q= p/(p-1)$.
\end{definition}

Test functions in Definitions  \ref{def_2s_on_ax_t}, \ref{def_2s_on_ax} have the following property.
\begin{lemma}
\label{eq:conv-test-func-surfaces}
For any $\varphi \in C(\bar G_T; L^q(\Omega, \mu))$, $1<q<\infty$, and any $t\in[0,T]$, it holds
\begin{align*}
%\label{eq:conv-test-func-surfaces}
    \lim_{\ve\to 0} & \,\, \ve \int\limits_{\Omega} 
   	\int\limits_{M_\ve(\omega)} |\varphi\left(t, x,T_{x'/\ve}\omega\right)|^q \,dS(x)\, dt \, d{\mathcal P}(\omega)=
   	\int\limits_{\Omega}  \int\limits_{G} |\varphi\left(t, x,\omega\right)|^q \,dt\, dx \, d\mu(\omega).
\end{align*}
\end{lemma}
\begin{proof}
By \eqref{eq:A_eps} we remove the discs that are too close to the boundary $\partial S$. Namely, there exists $\delta$ such that the random set $A_\ve(\omega)$ is contained in a compact subset $S_{\delta(\ve)} \Subset S$ with ${\rm dist}(S_{\delta(\ve)}, \partial S) = \delta(\ve)$, and $\delta(\ve) \to 0$ as $\ve \to 0$, e.g. $\delta =(d_0 + 2\overline r)\ve$.  
Changing variables $y'=x'/\ve$, using the Campbell formula, and rescaling back, we obtain:
\begin{align*}
&\ve \int\limits_{\Omega} \int \limits_0^T \int\limits_{0}^{L} 
\int\limits_{\partial A_\ve(\omega)} 
\chi_{S_{\delta(\ve)}}(x')|\varphi\left(t, x,T_{x'/\ve}\omega\right)|^q \,ds(x')\, dx_1 \, dt \, d{\mathcal P}(\omega)\\
&=\ve^2\int\limits_{\Omega} \int \limits_0^T \int\limits_{0}^{L} 
\int\limits_{\partial A(\omega)} \chi_{S_{\delta(\ve)}}(\ve y') |\varphi\left(t, x_1, \ve y',T_{y'}\omega\right)|^q \,ds(y')\, dx_1 \, dt \, d{\mathcal P}(\omega)\\
&= \ve^2 \int\limits_{\Omega} \int \limits_0^T \int\limits_{0}^{L} 
\int\limits_{\mathbf R^2} \chi_{S_{\delta(\ve)}}(\ve y') |\varphi\left(t, x_1, \ve y',\omega\right)|^q \,dy'\, dx_1 \, dt \, d\mu(\omega)\\
&=\int\limits_{\Omega} \int \limits_0^T \int\limits_{0}^{L} 
\int\limits_{\mathbf R^2} \chi_{S_{\delta(\ve)}}(x') |\varphi\left(t, x, \omega\right)|^q \,dx'\, dx_1 \, dt \, d\mu(\omega).
\end{align*}
Note that, in order to apply the Campbell formula, the compact subset $S_{\delta(\ve)}$ should be chosen independent of $\omega$. 
Passing to the limit, as $\ve \to 0$ completes the proof of the lemma. 
\end{proof}

Together with the weak stochastic two-scale convergence in the mean introduced above, we will need a strong version of it.
\begin{definition}\label{def:strong-conv-surface} We say that $u_\ve$ defined on $[0,T]\times M_\ve(\omega)$
converges \textbf{strongly} stochastically two-scale in the mean to %towards
$u\in L^p(G_T; L^p(\Omega, \mu))$ if $u_\ve$ converges stochastically two-scale in the mean to $u$ and 
%for any admissible $\varphi \in L^2(G\times\Omega)$ we have
\begin{align*}
   	\lim_{\ve\to 0} &\ve \int\limits_{\Omega} \int \limits_0^T
   	\int\limits_{M_\ve(\omega)} |u_\ve|^p dS(x)\, dt \, d{\mathcal P}(\omega) =
   	\int\limits_{\Omega}  \int\limits_{G_T} |u|^p \,dt\, dx \, d\mu(\omega).
\end{align*}
%\label{strong_on_ax}
\end{definition}

\begin{comment}
\begin{remark}
Use Lemma 3.1 from \cite{heida2022stochastic}, introduce $\cal T_\ve$ to get strong convergence, for all $t\in[0,T]$
\begin{align*}
\ve \int \limits_{\Omega} \int \limits_{M_\ve} |v_\ve(t,x,\omega) - (\mathcal T_\ve v_0)(t,x,\omega)|^2\, ds(x')\, d{\mathcal P} \to 0, \quad \ve \to 0.
\end{align*}
\end{remark}    
\end{comment}

%%%%%%%%%%%%%%%%%%%%%%%%%%%%%%%%%%%%%%%%%%%%%%%%%%%% 
\section{Problem setup}
\label{sec:setup}

Let $u_\ve=u_\ve(t,x,\omega)$ denote the electric potential in the bundle $G = G_\ve(\omega) \cup F_\ve(\omega)$, where $F_\ve(\omega), G_\ve(\omega)$, and $M_\ve(\omega)$ are defined by \eqref{eq:F_eps}--\eqref{eq:M_eps}, respectively. We assume that the conductivity is a nonlinear function of the electric field strength:
\begin{align*}
		u_\ve = \left\{
		\begin{array}{l}
			u_\ve^e \quad \mbox{in}\,\, G_\ve(\omega),\\[1mm]
			u_\ve^{i} \quad \mbox{in}\,\, F_\ve(\omega),
		\end{array}
		\right. \quad
        \sigma_\ve(x, \omega, |\nabla u_\ve|) = \left\{
		\begin{array}{l}
			\sigma^e(|\nabla u^e_\ve|)\quad \mbox{in}\,\, G_\ve(\omega),\\[1mm]
			\sigma^i(|\nabla u^i_\ve|) \quad \mbox{in}\,\, F_\ve(\omega).
		\end{array}
		\right.
\end{align*}
% We assume that the conductivity is a nonlinear function of the electric field strength:
% 	\begin{align*}
% 		\sigma_\ve(x, \omega, |\nabla u_\ve|) = \left\{
% 		\begin{array}{l}
% 			\sigma^e(|\nabla u^e_\ve|)\quad \mbox{in}\,\, G_\ve(\omega),\\[1mm]
% 			\sigma^i(|\nabla u^i_\ve|) \quad \mbox{in}\,\, F_\ve(\omega).
% 		\end{array}
% 		\right.
% 	\end{align*}
\begin{enumerate}[label=(H\arabic*)]
\setcounter{enumi}{3}
\item \label{H5}
The functions $\sigma^{e,i}(\eta)$ are continuously differentiable and satisfy the conditions
\begin{align*}
0< \underline{\sigma}\leq\sigma^{e,i}(\eta) \leq \overline{\sigma}, %\xi,  
\ \ \ \  
\underline{\sigma}< \eta\,\frac{d}{d\eta}\sigma^{e,i} (\eta ) + \sigma^{e,i}(\eta ), 
%\theta \xi^2 - C
 \ \ \ \  \forall\eta \in \mathbb{R}_+.
\end{align*}
\end{enumerate}
In Lemma \ref{lm:monotonicity} we prove that \ref{H5} implies monotonicity of $\sigma^e(|\xi|) \xi$, $\sigma^i(|\xi|) \xi$.

{
\begin{remark}
The choice of nonlinear conductivity depending on $|\nabla u|$ is motivated by the use of such models in oncology \cite{bihoreau2023mathematical}. 
Examples of conductivity functions satisfying assumption {(H4)} include sigmoid functions. A typical choice for tissue conductivity in electroporation models is a four-parameter sigmoid function \cite{jankowiak2020comparison}
\begin{align}
\sigma(\eta) = \sigma_0 + \frac{\sigma_1-\sigma_0}{2} (1 + \rm{erf}(k_{\rm ep}(\eta - E_{\rm th}))), \quad \eta \in \mathbb R_+.
\end{align}
Here, $\sigma_0$ is the conductivity of the non-electroporated tissue, $\sigma_1>\sigma_0$ is the conductivity of the fully electroporated tissue, $k_{\rm ep}$ is the slope of the nonlinearity, and $\rm{erf}$ is the Gauss error function.  
\end{remark}
}

The jump $v_\ve = [u_\ve] = u_\ve^{i} - u_\ve^e$ through {the lateral boundaries of axons 
$M_\ve(\omega)$ is called the transmembrane potential (or just membrane potential). 
On the membrane $M_\ve(\omega)$ we assume the current continuity {$\sigma^e(|\nabla u_\ve^e|)\nabla u_\ve^e\cdot \nu  = \sigma^i(|\nabla u_\ve^{i}|)\nabla u_\ve^i\cdot \nu$ ($\nu$ is the unit normal on $M_\ve(\omega)$ external to the axons)}, and {adopt the FitzHugh-Nagumo \cite{fitzhugh1955mathematical,nagumo1962active} model for the ionic current on $M_\ve(\omega)$:}
\begin{align*}
		%\label{eq:I_{\rm ion}}
		&I_{\rm ion}(v_{\ve}, g_{\ve}) = \frac{v^3_{\ve}}{3} - v_{\ve} - g_{\ve},\\
		&\partial_t {g}_{\ve} = \theta v_{\ve} + a - bg_{\ve}, \quad \theta, a, b >0,
\end{align*}
{
where $I_{\rm ion}(v_{\ve}, g_{\ve})$ is the ionic current, ${g}_{\ve}$ stands for the  recovery variable and $\theta$, $a$, $b$ are constants.
}  

The evolution of the electric potential $u_\ve=u_\ve(t, x, \omega)$ in $G_T=[0,T]\times G$ is described by the following coupled system for {$t\in [0,T]$}:
\begin{align}
		&\di{\sigma^i(|\nabla u_\ve|) \nabla u_\ve}  =0 \,\,& \mbox{in}\,\, & F_\ve(\omega), \nonumber \\
		&\di{\sigma^e(|\nabla u_\ve|) \nabla u_\ve}  =0 \,\,& \mbox{in}\,\, & G_\ve(\omega), \nonumber \\	
	    &\sigma^e(|\nabla u_\ve^e|)\nabla u_\ve^e\cdot \nu  = \sigma^i(|\nabla u_\ve|) \nabla u_\ve^{i}\cdot \nu\,\,& \mbox{on}\,\, &  M_\ve(\omega), \ \ \nonumber\\
		&v_\ve  = u^{i}_\ve-u^e_\ve\,\,& \mbox{on}\,\, &  M_\ve(\omega), \ \ \quad \quad \nonumber\\
		&\ve(c_m\partial_t v_\ve + I_{\rm ion}(v_\ve, g_\ve))
		= - \sigma^i(|\nabla u_\ve|) \nabla u_\ve^{i}\cdot \nu\,\,& \mbox{on}\,\, &  M_\ve(\omega), \ \ \quad \quad
		\label{eq:orig-prob} \\
		&\partial_t g_\ve  = \theta v_\ve + a - bg_\ve\,\,& \mbox{on}\,\, & M_\ve(\omega),  \nonumber \\
		&\sigma^e(|\nabla u_\ve^e|)\nabla u_\ve^e\cdot \nu   = J_\ve^e(t,x)\,\,& \mbox{on}\,\, & \Sigma, \nonumber \\
		%&\nabla u_\ve^e\cdot \nu   = 0, \, &x& \in \Gamma_\ve^m, \nonumber 
		%\\
		&u_\ve   = 0\,\,& \mbox{on}\,\, & S_{0} \cup S_L, \nonumber \\ 
		& v_\ve(0,x, \omega) =V_\ve(x,\omega), \,\,g_\ve(0,x,\omega)=G_\ve(x, \omega)\,\,& \mbox{on}\,\, & M_\ve(\omega), \nonumber
\end{align}
where $\nu$ denotes the external unit normal {on $M_\ve(\omega)$ (exterior to $F_\ve(\omega)$), the bases $S_0, S_L$, and $\Sigma$.} 
The function $J_\ve^e(t,x)$ represents an external {boundary} excitation of the nerve fascicle, and the constant $c_m$ is the membrane capacity. 
{Note that the only source of randomness is the distribution of axons in the cross section.}

We study the asymptotic behavior of $v_\ve$, as $\ve \to 0$, and derive a macroscopic model under the following conditions {on the initial and boundary data}:
\begin{enumerate}[label=(H\arabic*)]
\setcounter{enumi}{5}
\item \label{H6}
The initial data $V_\ve$ admits an extension in $G$ still denoted $V_\ve$, such that  $\|V_\ve\|_{H^1(G)} \le C$ and $V_\ve=0$ on $S_0\cup S_L$. We assume that $V_\ve(x,\omega)$ converges %strongly 
stochastically two-scale in the mean to $V_0\in L^2(G; L^2(\Omega, \mu))$, as $\ve\to 0$.
		
\item \label{H7}
The initial function $G_\ve(x, \omega)$ converges strongly stochastically two-scale in the mean %on the surface $M_\ve$  
to $G_0\in L^2(G; L^2(\Omega, \mu))$.
\vspace{2mm}
\item \label{H8}
The external excitation $J_\ve^e \in L^2((0,T)\times \Sigma)$ converges weakly to $J_0^e(t,x)$, as $\ve \to 0$, and
	\begin{align*}
		\int \limits_0^T %\int \limits_0^L 
  \int \limits_{\Sigma} |\partial_t J_\ve^e |^2\, dS(x) dt \le C.
	\end{align*}
\end{enumerate}

\begin{remark}
Assumption \ref{H6} together with Theorem 3.7 in \cite{bourgeat1994stochastic} guarantees strong stochastic two-scale convergence and that the limit function $V_0(x)$ does not depend on $\omega$. Moreover, due to Lemma \ref{lm:apr est aux}, the initial data satisfies
$$
  \ve \, \int\limits_{M_\ve(\omega)} V_\ve^4 \, dS(x) \le C.
$$ 
 
\end{remark}

%%%%%%%%%%%%%%%%%%%%%%%%%%%%%%%%%
Problem \eqref{eq:orig-prob} can be written as an evolution equation with a non-local operator on the membranes $M_\ve(\omega)$ of axons:
\begin{align}
\label{eq:orig-prob-nonlocal}
&\ve c_m \partial_t v_\ve + {\mathcal L}_\ve(t, v_\ve) + \ve I_{\rm ion}(v_\ve, g_\ve)=0 \,\,&\mbox{on}\,\, &(0,T)\times M_\ve(\omega), \nonumber\\
&\partial_t g_\ve  = \theta v_\ve + a - bg_\ve\,\,& \mbox{on}\,\, & (0,T)\times M_\ve(\omega),\\
&v_\ve(0,x, \omega) =V_\ve(x,\omega), \,\,g_\ve(0,x,\omega)=G_\ve(x, \omega)\,\,& \mbox{on}\,\, & M_\ve(\omega).\nonumber
\end{align}
The non-local integro-differential operator
\begin{align}
\label{def:oper-Leps}
{\mathcal L}_\ve(t, \cdot): D({\mathcal L}_\ve) \subset H^{1/2}(M_\ve(\omega)) \to H^{-1/2}(M_\ve(\omega))
\end{align} 
is defined, for $t\in[0,T]$, as follows: For any $\phi\in H^1(F_\ve(\omega) \cup G_\ve(\omega))$, $\phi|_{S_0\cup S_L}=0$, we set
\begin{align}
	\label{def:L_eps}
	&({\mathcal L}_\ve(t, v_\ve), [\phi])_{L^2(M_\ve(\omega))}
    %\nonumber\\ 
    %&\quad\quad\quad\quad 
    =\int \limits_{F_\ve(\omega) \cup G_\ve(\omega)}  \sigma_\ve(|\nabla u_\ve|) \nabla u_\ve \cdot \nabla \phi \, dx 	- \int \limits_{\Sigma} J_\ve^e \phi^e \, dS(x), 
	%	&\forall \ \phi\in H^1(F_\ve \cup G_\ve), \nonumber
\end{align}
where $\phi^e = \phi|_{G_\ve(\omega)}$, $\phi|_{S_0\cup S_L}=0$, and the function $u_\ve \in H^1(F_\ve(\omega) \cup G_\ve(\omega))$ in \eqref{def:L_eps}, for a given jump $[u_\ve]=v_\ve$, solves the following problem: 
\begin{align}
		\label{eq:stat-prob-0}
		&\di{\sigma_\ve(|\nabla u_\ve|) \nabla u_\ve}  =0 \,\,&\mbox{in}\,\, &F_\ve(\omega) \cup G_\ve(\omega), \nonumber\\
		&\sigma^e(|\nabla u_\ve^e|)\nabla u^e_\ve\cdot \nu  = \sigma^i(|\nabla u_\ve^i|) \nabla u_\ve^{i}\cdot \nu\,\,&\mbox{on}\,\, &M_\ve(\omega), \nonumber\\
		&v_\ve=u^{i}_\ve - u^e_\ve\,\,&\mbox{on}\,\, &M_\ve(\omega),
		\\
		&\sigma^e(|\nabla u_\ve^e|)  \nabla u_\ve^e\cdot \nu   = J_\ve^e(t, x)\,\,&\mbox{on}\,\, &\Sigma, \nonumber \\
		&u_\ve   = 0\,\,&\mbox{on}\,\, &S_{0} \cup S_L. \nonumber
\end{align}

%%%%%%%%%%%%%%%%%%%%%%%%%
\section{Effective problem and main result}
\label{sec:result}
To introduce the effective conductivity of the extracellular medium, we will use the notion of potential and solenoidal vector fields \cite{jikov2012homogenization}. 
{Recall that a vector $v\in \bigl(L^2(\Omega, {\mathcal P})\bigr)^2$ defined on $\Omega$ is called potential if almost all its realizations $v(T_{x'}\omega)$ are potential in $\mathbb{R}^2$, that is for almost all $\omega$, $v(T_{x'}\omega)$ admits the representation $v(T_{x'} \omega)=\nabla_{x'} \phi(x', \omega)$, $\phi(\,\cdot\,, \omega) \in H_{\rm loc}^1(\mathbb R^2)$, and $\int_\Omega v(\omega)\, d{\mathcal P}(\omega)=0$. An equivalent definition of the potential vectors can be found in Lemma 2.3, \cite{bourgeat1994stochastic}. The vector $v\in \bigl(L^2(\Omega, {\mathcal P})\bigr)^2$ is called solenoidal if almost all its realizations $v(T_{x'}\omega)$ are solenoidal in $\mathbb{R}^2$, i.e. 
		\begin{align*}
			\int_{\mathbb R^2} v(T_{x'} \omega) \cdot \nabla \phi\, dx' =0, \quad \phi\in C_0^\infty(\mathbb R^2). 
	\end{align*}
	We denote by $V^2_{\rm pot} (\Omega, {\mathcal P})$ the set of potential vectors with zero mean value, and by $L^2_{\rm sol} (\Omega, {\mathcal P})$ the set of all solenoidal vectors. Then the orthogonal Weyl decomposition holds (Lemma 7.3, \cite{zhikov1993averaging}):
$$
\bigl(L^2(\Omega, {\mathcal P})\bigr)^2=V^2_{\rm pot} (\Omega, {\mathcal P}) 	{\displaystyle \oplus } L^2_{\rm sol} (\Omega, {\mathcal P}).
$$
We define the effective extracellular conductivity $\sigma^e_{\rm hom}$  as follows
\begin{align}
\label{eff tensor}
\sigma^e_{\rm hom}(\xi)=\nabla\Phi(\xi), \quad \xi\in \mathbb{R}^3,
\end{align}
where 
\begin{align}
\label{pot}
  \Phi(\xi)= \inf_{w\in V^2_{\rm pot} (\Omega)} \,
		\int_{\Omega\setminus {\cal A}} Q^e(|\xi+(0,w)|) \, d{\cal P}(\omega), 
 \quad Q^e(\eta)=\int_0^\eta \sigma^e(\zeta)\zeta \, d\zeta.
	\end{align}
{The definition \eqref{eff tensor} can be equivalently written as 
\begin{align}
\label{know_how}
    \sigma^e_{\rm hom}(\xi)=\int_{\Omega \setminus \mathcal{A}} \sigma^e(|\xi + (0,w)|) (\xi + (0,w))d{\cal P}(\omega) 
\end{align}
with $w$ minimizing \eqref{pot}.
}    

We will approximate the intracellular potential $u_\ve^i$ by a function $u^i(t,x,\omega)$ such that almost all its realizations are constants on the axons (see problem \eqref{eq:hom-prob}). To this end, we will need the following class of functions:
\begin{align}
\label{def:K(A)}
{\cal K}({\cal A})= \big\{w\in L^2({\cal A}, {\mathcal P}) :\exists \tilde{w}\in H^1(\Omega, {\mathcal P}) \,\, \mbox{s.t.}\,\, \tilde{w}|_{\A }=w, \ \ \nabla_\omega \tilde{w}=0 \ \text{in}\  \A  \big\}.
\end{align}
For any $w\in {\cal K}({\cal A})$ it holds that $w\in L^2(\Omega, \mu)$ (see Lemma 6.2 in \cite{piatnitski2020homogenization}), and we can introduce the orthogonal projector ${\bf P}_{\cal K(\A )}$  on ${\cal K}({\cal A})$ in $L^2(\Omega,\mu)$. We denote by ${\cal K}_\mu({\cal A})$ the Hilbert space $L^2(\Omega, \mu)\cap{\cal K}({\cal A})=:{\cal K}_\mu({\cal A})$.  

\begin{remark}
For any function $w\in {\cal K}({\cal A})$, $\nabla_\omega \tilde w(T_{y'}\omega) = \nabla_{y'} \tilde w(T_{y'}\omega)=0$ for $y'\in A(\omega)$ and a.e. $\omega$. That is almost all realizations $w(T_{y'}\omega)$ are constants in $y'$ on the axons $A_j(\omega)$.  
\end{remark}

The lemma below shows that on functions $w\in \cal{K(A)}$ the $L^2(\A , {\mathcal P})$-norm is equivalent to the $L^2(\Omega, \mu)$-norm.
\begin{lemma} \label{lemma_radii}
Let $f\in L^1(\Omega)$ and $
f(T_{x'}\omega)=f_j(\omega) \,\, \mbox{is constant in}\,\, x' \,\, \mbox{in} \ \bar A_j(\omega), \ j=1, 2, \ldots$,
for a.a. $\omega\in \bar \Omega$ (with respect to both measures $d{\cal P}(\omega)$ and $d\mu (\omega))$.
Then 
\begin{align}
\label{eq:formulka}
\int\limits_{\A }f(\omega)d{\cal P}(\omega)=\frac{1}{2}\int\limits_{\Omega}r(\omega) f(\omega) d\mu (\omega).
\end{align}
\begin{proof}
It suffices to prove this lemma for 
$f\ge 0$ such that $f=0$ in $\Omega\setminus \bar \A $. For large enough $N>0$ consider the squares $q=\bigl[-N, N\bigr]^2$ and $q_{\bar r}=\bigl[-N-2\bar r, N+2\bar r\bigr]^2$ (see the assumption \ref{H3}).  
%Wright down Campbell's formula for $q_{\bar r}$
{Since $T_{x'}$ is measure preserving, 
the Campbell formula \eqref{Campbell} yields}
\begin{align}
&\int\limits_{\A } \int\limits_{\mathbb{R}^2} 
\chi_{q_{\bar r}}(x')f(\omega) dx'd{\cal P}(\omega)=
\int\limits_{\Omega}\int\limits_{\mathbb{R}^2}
    \chi_{\A }(T_{x'}\omega)\chi_{q_{\bar r}}(x')f(T_{x'}\omega) dx'd{\cal P}(\omega) \notag\\ 
&{=\int\limits_{\Omega}\int\limits_{A(\omega)\cap q_{\bar r}}
f(T_{x'}\omega) dx' d{\cal P}(\omega)} =
\int\limits_{\Omega}\sum\limits_{j} f_j(\omega)   |A_j(\omega)\cap q_{\bar r}|d{\cal P}(\omega)\notag\\ 
&\geq
\int\limits_{\Omega}\sum\limits_{j:A_j(\omega)\cap q \not=\emptyset }    f_j(\omega) \pi r^2_j(\omega)d{\cal P}(\omega)
%=\frac{1}{2}\int\limits_{\Omega}d{\cal P}(\omega) \sum\limits_{j:A_j(\omega)\cap q_R \not=\emptyset }    f_j(\omega) 2\pi r_j(\omega) \cdot r_j(\omega)
\notag \\
&\geq \frac{1}{2}\int\limits_{\Omega}
   \int\limits_{\partial A(\omega)} \chi_{q}(x')f(T_{x'}\omega) r(T_{x'}\omega) ds(x') 
 \, d{\cal P}(\omega)\notag \\
&= \frac{1}{2}\int\limits_{\mathbb{R}^2}\int\limits_{\Omega}
     \chi_{q}(x')f(\omega) r(\omega) \, d\mu(\omega) \, dx'.
\label{23}
\end{align}
Calculating the integrals in $x'$ in both parts of the inequality \eqref{23}, we get
\begin{equation}
(2N+4\bar r)^2 \int\limits_{\Omega} f(\omega) d{\cal P}(\omega)
\geq \frac{(2N)^2}{2}\int\limits_{\Omega} f(\omega) r(\omega) d\mu(\omega).
\label{24}
\end{equation}
{Analogously, starting with $\chi_{q(x')}$ instead of $\chi_{q_{\bar r}(x')}$ in \eqref{23}, we obtain the estimate from above}
% \begin{align}
% \int\limits_{\Omega}\int\limits_{\mathbb{R}^2}
%     \chi_{q(x')}\chi_{\A }(T_{x'}\omega)f(T_{x'}\omega) dx'd{\cal P}(\omega)
%  %   =\int\limits_{\mathbb{R}^2} \int\limits_{\Omega}
%   %  \chi_{q_{\bar r}}(x')\chi_{\A }(\omega)f(\omega) dx'd{\cal P}(\omega).
%    =\int\limits_{\mathbb{R}^2} \int\limits_{\A } \int\limits_{\mathbb{R}^2} 
%    \chi_{q(x')f(\omega) dx'd{\cal P}(\omega),
% \label{Camp_qua} 
% \end{align}
%using the formula \eqref{Camp_qua} on $q$, 
%we obtain
\begin{equation}
 (2N)^2 \int\limits_{\Omega} f(\omega) d{\cal P}(\omega)\le
 \frac{(2N+4\bar r)^2}{2}\int\limits_{\Omega} f(\omega) r(\omega) d\mu(\omega).
\label{25}
\end{equation}
Then, passing to the limit as $N\to\infty$ in \eqref{24}, \eqref{25} completes the proof.
\end{proof}
\end{lemma}

To formulate the limit problem \eqref{eq:hom-prob}, we introduce the functional spaces 
\begin{align}
&{\cal H}_e=\{ u^e\in H^{1}(G): u^e|_{S_0\cup S_L}=0 \}, \notag\\
&{\cal H}_i=\{ u^i\in L^2(G,{\cal K_\mu(\A )})\ : \partial_{x_1} u^i\in L^2(G,{\cal K_\mu(\A )}), \ u^i|_{S_0\cup S_L}=0 \},  \notag\\
&{\cal H}'_i=\{v=\partial^2_{x_1 x_1} u^i, u^i\in {\cal H}_i\},\notag
\end{align}
where ${\cal H}'_i$ is dual to ${\cal H}_i$.

Consider the non-local integro-differential operator
\begin{align}
{\mathcal L}_{\rm hom}(t, \cdot): {\cal H}_i \to \
{\cal H}'_i 
\end{align} 
given for $t\in[0,T]$ by
\begin{align}
\label{def:oper-Lhom}
{\mathcal L}_{\rm hom}(t, v)=-\frac{r(\omega)}{2}\partial_{x_1} 
\left( \sigma^i\bigl(|\partial_{x_1}(v+u^e)|\bigr) \partial_{x_1}(v+u^e)\right),
\end{align}
where $u^e\in {\cal H}_e$ solves
\begin{align}
\label{pr_u^e}
&\frac{1}{2}\int_\Omega\int_{G} r(\omega)\sigma^i\left(|\partial_{x_1} (v+u^e)|\right) \partial_{x_1} (v+u^e) \partial_{x_1} \phi^e \, dx\, d\mu(\omega) \notag \\
&
+\int_{G} \sigma_{\rm hom}^e(\nabla u^e)  \cdot \nabla \phi^e \, dx
- \int_{\Sigma} J_0^e \phi^e \, dS(x)=0, \quad \forall 
\phi^e\in {\cal H}_e.
\end{align}
%Notice that from the definition of ${\mathcal L}_{\rm hom}$ it follows that
It follows from the definition of ${\mathcal L}_{\rm hom}$ that
\begin{align}
		\label{def:L_hom}
		&({\mathcal L}_{\rm hom}(t, v), [\phi])_{L^2(G; L^2(\cal M, \mu))}= 
        \int_{G} \sigma_{\rm hom}^e(\nabla u^e)  \cdot \nabla \phi^e \, dx\\
        & +\frac{1}{2}\int_\Omega\int_{G} r(\omega)\sigma^i(|\partial_{x_1} u^i|) \partial_{x_1} u^i \partial_{x_1} \phi^i \, dx\, d\mu(\omega)
        - \int_{\Sigma} J_0^e \phi^e \, dS(x).\nonumber
\end{align}
holds for $u^e$ solving \eqref{pr_u^e}, $u^i=v+u^e $ and for any
$\phi^e\in {\cal H}_e, \phi^i \in {\cal H}_i$, $[\phi]=\phi^i-\phi^e$.

We will prove that the asymptotic behavior of solutions of {the boundary value problem} \eqref{eq:orig-prob} is described by the following limit  model: 
\begin{align}
\label{eq:hom-prob}
& c_m\partial_t v_0 
+{\mathcal L}_{\rm hom}(t, v_0)+ {\bf P}_{\cal K(\A )} I_{\rm ion}(v_0, g_0)=0
%= \frac{r(\omega)}{2} \, \partial_{x_1}\left( \sigma^i(|\partial_{x_1} u_0^i|) \partial_{x_1} u^{i}_0\right) \,\,
&\mbox{in}\,\,  &G_T\times \cal M,\nonumber\\
%& \int_{\cal M} \left(c_m\partial_t v_0 + {\bf P}_{\cal K(\A )}I_{\rm ion}(v_0, g_0)\right) d\mu(\omega) = - \, \di{\sigma^e_{hom} (\nabla u_0^e)}\,\,&\mbox{in}\,\, &G,\nonumber\\
&\partial_t g_{0}  = \theta v_{0}+a - b\, g_{0} \,\,&\mbox{in}\,\, &G_T\times \cal M, \\
%&v_{0}=u_{0}^{i}- u_0^e\,\,&\mbox{in}\,\, &G\times \cal M,\ \\ 
%&u_0^e = u_{0}^{i}=0\,\,&\mbox{on}\,\, &(S_0\cup S_L)\times\cal M,\nonumber\\
%&\sigma^e_{hom} (\nabla u_0^e) \ \cdot \nu = J_0 \,\,&\mbox{on}\,\, &\Sigma ,\nonumber\\
&v_{0}(0,x) =V_0(x), \,\, g_{0}(0,x,\omega)=G_0(x,\omega)\,\,&\mbox{in}\,\,  &G\times \cal M,\nonumber
\end{align}
with $u_{0}^{e}=u_{0}^{e}(t,x)$, $u_{0}^{i}=u_{0}^{i}(t,x,\omega)$ such that $u_{0}^{i}(t,x,\cdot):{\cal M} \to {\cal K(\A )}$. Here with a slight abuse of notation, we set
${\bf P}_{\cal K(\A )} I_{\rm ion}(v_0, g_0)=\frac{v^3_{0}}{3} - v_{0} - {\bf P}_{\cal K(\A )}g_{0}$.
The well-posedness of the limit problem \eqref{eq:hom-prob} will be shown in Section \ref{sec:existence-eff-prob}.

{The main result of the paper describing the convergence of solutions of \eqref{eq:orig-prob} to those of the macroscopic problem \eqref{eq:hom-prob} is contained in the following theorem.}
\begin{theorem}
		\label{th:main-short} 
		Under assumptions  \ref{H1}--\ref{H8}, for $t\in [0,T]$, the solution $\{v_\ve, g_\ve\}$ of the microscopic problem \eqref{eq:orig-prob} converges strongly stochastically two-scale in the mean {on the random surfaces $M_\ve(\omega)$} to the solution $\{v_{0}$, $g_{0}\}$ of the macroscopic problem \eqref{eq:hom-prob}, as $\ve\to 0$. Moreover, the potential $u_\ve$ converges stochastically two-scale in the mean in $G_T$  (see Definition \ref{def:2scale-G}) to 
  \begin{align*}
  u_0(t,x,\omega)=\chi_{\Omega\setminus {\bar\A }} (\omega)\, u^e_0 (t,x) \, + \chi_{\bar \A } (\omega)\, u^i_0(t,x,\omega),
  \end{align*}
   where  $u^e_0$ is a solution of the problem \eqref{pr_u^e} with $v=v_0$, $u^i_0=v_0+u_0^e$.
\end{theorem}
%%%%%%%%%%%%%%%%%%%%%%
Note  that the limit problem \eqref{eq:hom-prob} can be formulated in the following form:
\begin{align*}
	\label{eq:hom-prob-system}
	& c_m\partial_t v_{0} + {\bf P}_{\cal K(\A )} I_{ion}(v_{0}, g_{0})  = \frac{r(\omega)}{2} \partial_{x_1}( \sigma^i(| \partial_{x_1} u^{i}_0|))  & &\mbox{in}\, G\times {\mathcal M},\nonumber\\
	& \int \limits_\Omega \left(c_m\partial_t v_{0} + {\bf P}_{\cal K(\A )} I_{ion}(v_{0}, g_{0})\right)\, d\mu(\omega)  = -\di{\sigma_{hom}^e(\nabla u_0^e)}  & &\mbox{in}\, G,\nonumber\\
	%\label{eq:hom-prob}
	&\partial_t g_{0}  = \theta v_{0}+a - b\, g_{0}  & &\mbox{in}\, G\times {\mathcal M},\\
	&v_{0}=u_{0}^{i} - u_0^e & &\mbox{in}\, G\times {\mathcal M},\ \nonumber\\ 
	&u_0^e = u_{0}^{i}=0& &\mbox{on}\, S_0\cup S_L,\nonumber\\
	&\sigma_{hom}^e(\nabla u_0^{e}) \cdot \nu = J_0^e &&\mbox{on}\,  \Sigma,\nonumber\\
	&v_{0}(0,x) =V_0(x), \,\, g_{0}(0,x,\omega)=G_0(x,\omega)  
	\hskip -0.5cm & &\mbox{in}\, G\times {\mathcal M }.\nonumber
\end{align*}

%%%%%%%%%% Expansions:

{ 
\begin{remark} To discuss the essence  %origin 
of the effective bidomain operator and asymptotics of potentials consider a simplified problem.
Disregard the recovery variable $g_\ve$ and 
consider a stationary counterpart of problem \eqref{eq:orig-prob} obtained 
by replacing 
$\ve (c_m\partial_t v_\ve+I_{ion}(v_\ve,g_\ve))$ by $\ve f(x)$ with a known function $f(x)$ (independent of $\ve$). 
Then, problems for the extracellular potential $u_\ve^e$ and intracellular potential $u_\ve^i$ decouple so that the asymptotic behavior of these potentials can be studied separately.
Indeed, 
$u^e_\ve$ and $u^i_\ve$ can be found via the 
minimization problems.  
\begin{equation}
\label{Min_ue}
\min_{u_\ve^e} 
\int_{G_\ve(\omega)} Q^e(|\nabla u_\ve^e|)dx
-\int_{\Sigma} J_\ve^e u_\ve^e dS
-\ve \int_{M_\ve(\omega)} f u_\ve^e dS, 
\end{equation}
\begin{equation}
\label{Min_ui}
\min_{u_\ve^i} 
\int_{F_\ve(\omega)} Q^i(|\nabla u_\ve^i|)dx
+\ve \int_{M_\ve(\omega)} f u_\ve^i dS,
\end{equation}
where 
%$Q^e(\eta)$ is defined in \eqref{pot}, and 
$Q^\alpha(\eta)=\int_0^\eta \sigma^\alpha(\zeta)\zeta d\zeta$, $\alpha=i, e$ (cf. \eqref{pot}). Considering first $u_\ve^e$, we postulate the ansatz   
\begin{equation}
\label{ansatz}
u_\ve^e=u_0^e(x)
+\ve u_1^e(x,\omega, x^\prime/\ve)+o(\ve),
\end{equation}
assuming that  $w(x, T_{y'}\omega)=\nabla_{y^\prime} u_1^e(x,\omega, y^\prime)$ is a stationary random vector field. 
%Then 
%$u^e_\ve$ can be retrieved via the 
%minimization problem  
%$$
%\min 
%\int_{G_\ve(\omega)} Q^e(|\nabla u_\ve^e|)dx
%-\int_{\Sigma} J_0^e u_\ve^e ds
%-\ve \int_{M_\ve(\omega)} f u_\ve^e ds, 
%$$
Then substituting the ansatz \eqref{ansatz} 
 in \eqref{Min_ue}, one can formally derive the effective conductivity 
\eqref{eff tensor}--\eqref{pot}. 
Indeed, since 
$\nabla u_\ve^e(x)\simeq\nabla u_0^e(x)+(0, w(x, T_{x^\prime/\ve} \omega))$ we get,
\begin{align}
\label{how-to}
\int_\Omega \int_{G_\ve(\omega)} Q^e(|\nabla u_\ve^e|)dx d{\cal P}(\omega)
\to 
\int_G \int_{\Omega\setminus \mathcal{A}}  Q^e(|\nabla u_0^e(x)+(0,w)|) d{\cal P}(\omega) dx.
\end{align}
Passing to the limit in other terms, then minimizing in $w$ we get 
\begin{align*}
    &\inf_{u_0^e} \int_G \Big[\inf_{w \in V_{pot}} \int_{\Omega\setminus \mathcal A} Q^e(|\xi + w|)d\mathcal{P}(\omega) \Big|_{\xi=\nabla u_0^e}- \int_{\Sigma} J_0^e u_0^e dS- \int_\Omega f u_0^e d\mu \Big] dx\\
    &= \inf_{u_0^e} \int_G \Big[ \Phi(\nabla u_0^e)- \int_{\Sigma} J_0^e u_0^e dS  - \int_\Omega f u_0^e d\mu \Big]dx.
\end{align*}
Finally, taking the variation in $u_0^e$ we arrive at the homogenized  
equation for the limit function $u_0^e$:  $-{\rm div}\,\sigma^e_{\rm hom}(\nabla u_0^e)=\int_\Omega f(x) d\mu(\omega)$. 

% \ip{Minimizing in $u_0^e$ and $w$ (independently, one at a time) we get 
% \begin{align*}
%     &\inf_{u_0^e} \int_G \Big[\inf_{w \in V_{pot}} \int_{\Omega\setminus \mathcal A} Q^e(|\xi + w|)\mathcal P \Big|_{\xi=\nabla u_0^e}\,- \int_\Omega f u_0^e d\mu \Big] dx\\
%     &= \inf_{u_0^e} \int_G \Big[ \Phi(\nabla u_0^e)\,  - \int_\Omega f u_0^e d\mu \Big]dx
% \end{align*}
% which is equivalent to (variation=0) 
% \begin{align*}
%     \int_G \nabla \Phi(\nabla u_0^e)\cdot \nabla \varphi\, dx = \int_G f\varphi dx.
% \end{align*}
% This looks as (5.1)-(5.2).

% On the other hand, if I first take a variation in $u_0^e$ in \eqref{how-to}, I get:
% \begin{align}
% \label{how-to-2}
%     \int_G \int_{\Omega \setminus \mathcal{A}} \sigma^c(|\nabla u_0^e(x) + (0,w)|)(\nabla u_0^e(x) + (0,w)) \cdot \nabla \phi(x)\,d{\cal P}(\omega) dx=\int_\Omega f\, d\mu. 
% \end{align}
% And if I compute the variation of \eqref{how-to} in $w$, I get a similar homogeneous eq-n for $w$ which is the cell problem. Is this equivalent to (5.1)-(5.2)? I thought that one could not compute explicitly $\nabla \Phi(\xi)$, because the minimizer $w$ in (5.2) depends on $\xi$, so one should then differentiate the minimizer it w.r.t $\xi$. So, the question is if we can compute $\nabla \Phi(\xi)$ and get \eqref{how-to-2}.}

Since axons form disconnected structure of
thin cylinders, the intracellular potential is approximated (on the fine scale) by a random stationary function constant in cross sections of axons:    
\begin{equation}
    u_\varepsilon^i=u_0^i(x_1,\ve x_j^\prime(\omega),T_{x^\prime/\ve}\omega)+o(\ve),
\label{razlozhenieui}
\end{equation}
where $x_j'(\omega)$ are the coordinates of the axons' center, and $u_0^i(x_1,x^\prime,\,\cdot\,)\in \mathcal{K}(\mathcal{A})$. Then one can show,
using the Campbell formula, that 
\begin{equation}
\int_\Omega \int_{F_\ve(\omega)} Q^i(|\nabla u_\ve^i|)dx d{\cal P}(\omega)
\to \int_\Omega \int_{G}\frac{r(\omega)}{2} Q^i(|(\partial_{x_1} u_0^i|)dx d{\mu}(\omega).
\end{equation}
Passing to the limit in the second term of \eqref{Min_ui} and then taking variations in $u_0^i$ 
we formally arrive in the stochastic part of the bidomain operator,  
$\frac{r(\omega)}{2} 
\partial_{x_1}( \sigma^i(| \partial_{x_1} u^{i}_0|))=f(x)
$.

The formal computations above are used for the construction of test functions \eqref{def phi^e_eq}--\eqref{def phi^e_bc2}. While passing to the limit, we combine the method of  monotone operators combined with the convergence of the test functions (see Lemma \ref{lm:conv phi^e}).
\end{remark}
}

%%%%%%%%%%

\begin{remark}
{As mentioned in the introduction, in the numerical study \cite{mandonnet2011role}, the authors modeled an axonal fasicle by means of a linear bidomain model. The aim of the study
was to provide a computational model to estimate the
regions in which electrical stimulation would lead to a membrane potential exceeding a certain threshold and, in this way, generate a propagating signal. Using the macroscopic multidomain type model \eqref{eq:hom-prob} would allow us to account for axons of different radii without discretizing the original three-dimensional geometry of the bundle.}
\end{remark}
%%%%%%%%%%%%%%%%%%%%%%%%%%%%%%

 %%%%%%%%%%%%%%%%%%%%%%%%%%
\section{Proof of Theorem \ref{th:main-short}}
\label{sec:proof}
{The proof of Theorem \ref{th:main-short} is organized as follows. In Section \ref{sec:micro} we derive a priori estimates (Lemma \ref{lm:apriori-est}), construct an extension operator from the extracellular domain $G_\ve(\omega)$ to the whole bundle $G$ (Lemma \ref{lm:extension-u^e}), as well as construct a quasi-extension of $u_\ve^i$ in $G$. In Section \ref{sec:pass-to-lim}, we formulate the key compactness results to extract converging subsequences for the membrane potential, the intra- and extracellular potentials, and their derivatives. Lemma \ref{lm:conv_u_ds} is another important result of the section, stating the convergence of the traces on the lateral boundary of the axons. Finally, we pass to the limit using the Minty method for monotone operators. Note that a special choice of test functions in the extracellular domain (see problem \eqref{def phi^e_eq}--\eqref{def phi^e_bc2} and Lemma \ref{lm:conv phi^e}) leads to the appearance of the effective conductivity. In Section \ref{sec:examples_random_geometry} we present several random geometries satisfying assumptions \ref{H1}--\ref{H4}.}

\subsection{Microscopic problem}
\label{sec:micro}
We start with writing down a weak formulation of the microscopic problem. For brevity, we skip the argument $\omega$ in $v_\ve(t,x), g_\ve(t,x)$. In order to transform the problem to a one with a monotone operator, we introduce a parameter $\lambda>0$, which will be chosen later, and make the following factorization
\begin{align*}
v_\ve = e^{\lambda t}\hat v_\ve, \quad 
g_\ve = e^{\lambda t} \hat g_\ve,\quad
{\hat{\sigma}_\ve({t, x, \omega,} \eta)}= \begin{cases}
    \hat \sigma^e(t, \eta)=\sigma^e(e^{\lambda t}\eta) \,\, \mbox{in} \,\,G_\ve(\omega)\\ \hat \sigma^i(t, \eta)=\sigma^i(e^{\lambda t}\eta) \,\, \mbox{in}\,\, F_\ve(\omega).
    \end{cases}
\end{align*}
We define the operator $\hat{{\mathcal L}}(\hat{v}_\ve)=\hat{{\mathcal L}}_\ve(t,\omega; \hat v_\ve)$ as follows. For any $\phi\in H^1(F_\ve(\omega) \cup G_\ve(\omega))$, $\phi|_{S_0\cup S_L}=0$, we set
	\begin{align}
		\label{def:hat_L_eps}
		&(\hat{\mathcal L}_\ve(\hat v_\ve), [\phi])_{L^2(M_\ve(\omega))} \nonumber\\ 
		&=\int \limits_{F_\ve(\omega) \cup G_\ve(\omega)}  \hat\sigma_\ve(t, x, \omega, |\nabla \hat u_\ve|) \nabla\hat u_\ve \cdot \nabla \phi \, dx 	- \int \limits_{\Sigma} e^{-\lambda t}J_\ve^e \phi^e \, dS(x), 
		%	&\forall \ \phi\in H^1(F_\ve \cup G_\ve), \nonumber
	\end{align}
	where $\phi^e = \phi|_{G_\ve(\omega)}$, $\phi|_{S_0\cup S_L}=0$, and the function $\hat u_\ve \in H^1(F_\ve(\omega) \cup G_\ve(\omega))$ in \eqref{def:hat_L_eps}, for a given jump $[\hat u_\ve]=\hat v_\ve$, solves the following problem: 
	\begin{align}
		\label{eq:hatprob}
		&\di{\hat \sigma_\ve(t, x, \omega, |\nabla\hat  u_\ve|) \nabla \hat u_\ve}  =0 \,\,&\mbox{in}\,\, &F_\ve(\omega) \cup G_\ve(\omega), \nonumber\\
		&\hat \sigma^e(t,|\nabla \hat  u_\ve^e|)\nabla\hat  u^e_\ve\cdot \nu  = \hat \sigma^i(t,|\nabla \hat  u_\ve^i|) \nabla \hat  u_\ve^{i}\cdot \nu\,\,&\mbox{on}\,\, &M_\ve(\omega), \nonumber\\
		&\hat v_\ve=\hat u^{i}_\ve -\hat  u^e_\ve\,\,&\mbox{on}\,\, &M_\ve(\omega),
		\\
		&\hat \sigma^e(t,|\nabla \hat u_\ve^e|)  \nabla \hat u_\ve^e\cdot \nu   =  e^{-\lambda t}J_\ve^e(t, x)\,\,&\mbox{on}\,\, &\Sigma, \nonumber \\
		&\hat u_\ve   = 0\,\,&\mbox{on}\,\, &S_{0} \cup S_L. \nonumber
	\end{align}
The original system \eqref{eq:orig-prob-nonlocal} transforms into
\begin{align}
\label{eq:orig-prob-nonlocal-hat}
&\ve c_m \partial_t \hat{v}_\ve + \hat{{\mathcal L}}_\ve( \hat{v}_\ve) + \ve \hat{I}_{\rm ion}(\hat{v}_\ve, \hat{g}_\ve)=0 \,\,&\mbox{on}\,\, &(0,T)\times M_\ve(\omega), \nonumber\\
&\partial_t \hat{g}_\ve  = \theta \hat v_{\ve} + ae^{-\lambda t} - (b+\lambda)\hat g_{\ve}\,\,& \mbox{on}\,\, & (0,T)\times M_\ve(\omega),\\
&\hat{v}_\ve(0,x, \omega) =V_\ve(x,\omega), \,\,\hat{g}_\ve(0,x,\omega)=G_\ve(x, \omega)\,\,& \mbox{on}\,\, & M_\ve(\omega),\nonumber
\end{align}
where
\begin{align}
\label{def:I_ion_hat}
	\hat I_{\rm ion}(\hat v_{\ve}, \hat g_{\ve}) =\frac{ e^{2\lambda t}}{3} \hat v^3_{\ve} +  \left(c_m\lambda -1\right)\hat v_\ve - \hat g_\ve.
\end{align}
Choosing $\lambda$ such that 
				\begin{align}
					\label{lambda}
					\lambda \ge \frac{1}{2} (\theta+3)\quad \mbox{and} \quad \lambda \ge \frac{1}{2} (\theta+1) -b
				\end{align} 
and using Lemma \ref{lm:monotonicity}, we ensure that the problem is monotone.
{The proof of the monotonicity and well-posedness of problem \eqref{eq:orig-prob-nonlocal-hat}, for a fixed $\omega$, follows the lines of the proof in the deterministic case (see  \cite{JEREZHANCKES2023103789}). 
We refer, namely, to Theorem 1.4 in \cite{lions1969quelques} and Remark 1.8 in Chapter 2 (see also Theorem 4.1 in \cite{showalter2013monotone}).}
	
	\begin{lemma}[A priori estimates]
		\label{lm:apriori-est}
		Assume hypotheses \ref{H5}--\ref{H8} hold. A solution  $W_\ve=(\hat v_\ve, \hat g_\ve)^T$ of problem \eqref{eq:orig-prob-nonlocal-hat},  for any $t \in [0,T]$, satisfies the following estimates:
		\begin{align}
			\label{apr_est_v}
			&\ve \int\limits_{M_\ve(\omega)}\ |\hat v_\ve|^4 \, dS(x)
			+ \ve \int\limits_0^T \int\limits_{M_\ve(\omega)} |\partial_t \hat v_\ve|^2 \, dS(x) \, dt\le C,\\
			\label{apr_est_g}
			&\ve \int\limits_{M_\ve(\omega)} |\hat g_\ve|^2 \, dS(x)
			+ \ve \int\limits_0^T \int\limits_{M_\ve(\omega)} |\partial_t \hat g_\ve|^2 \, dS(x) \, dt\le C,
		\end{align}
		with a constant $C$ independent of $\ve$ and t, but depending on $T$, and the norms of initial functions $\|G_\ve\|_{L^2(M_\ve(\omega))}$, $\|V_\ve\|_{H^1(G)}$.
\end{lemma}

Denoting $\hat u_\ve=e^{-\lambda t}u_\ve$ with $[\hat u_\ve]=\hat v_\ve$, one can obtain a priori estimates for the electric potential in the intra- and extracellular domains, as well as its traces on the membrane:
		\begin{align}
			\label{apr_est_u}
			&\int\limits_{F_\ve(\omega)\cup G_\ve(\omega)} \biggl( |\nabla \hat u_\ve|^2\, + |\hat u_\ve|^2 \biggr) dx \le C,\\
			%\int_{\Sigma} |_\ve|^2\, ds |_{t=0} \le C,
			\label{apr_est_ubound}
			&\ve \int_{M_\ve(\omega)} \biggl(|\hat u^e_\ve|^4  + |\hat u^{i}_\ve|^4 \biggr)\, dS(x) \le C
		\end{align}
for all $t\in [0,T]$.
The inequalities \eqref{apr_est_v}, \eqref{apr_est_g}, 	\eqref{apr_est_u} one can obtain in the same way as in \cite{JEREZHANCKES2023103789}.

{To prove \eqref{apr_est_ubound} we use  estimate \eqref{apr_est_v} for the potential $\hat v_\ve=\hat u^i_\ve-\hat u^e_\ve$ combined with the bound $\int\limits_{M_\ve(\omega)}\ |\hat u_\ve^e|^4 \, dS(x)\leq C/\ve$. The latter bound is derived from the 
$H^1$-estimate \eqref{apr_est_u}
with the help of Lemmas \ref{lm:apr est aux} and \ref{lm:extension-u^e} below.} 
\begin{lemma}
\label{lm:apr est aux}
For any function $u(x)\in H^1(G)$ the following estimate holds 
		\begin{align*}
			%\label{est:t=0}
			\ve \int_{M_\ve(\omega)} |u|^4\, dS(x)  \le C 
			\biggl(\ve \|u\|^4_{H^1(G)}  + \|u\|^4_{L^4(G)} \biggr) .
		\end{align*}
	\end{lemma}
The proof is given in Lemma 4.1 \cite{pettersson2023bidomain} where we treat the case of identical randomly distributed axons. A similar argument is presented in \cite{piatnitski2020homogenization}, Lemma 8.1. The main idea is to derive an estimate for one fixed disc, rescale by $\ve$, and sum up over all the discs to obtain the integral over $M_\ve(\omega)$.
	
With the functions $\hat u^e, \hat u^i$ defined in the extra- and intracellular domains, we 
associate the functions $\Pi_\ve^\omega \hat{u}_\ve^e$ and $\tilde{u}_\ve^i$ in the whole $G$ constructed in the following way. The function $\Pi_\ve^\omega u_\ve^e$ is a classical extension of $\hat{u}_\ve^e$ on $G$, while $\tilde{u}_\ve^i$ equals to the mean value of $\hat{u}_\ve^i$ on the cross sections of axons and is extended to $G$ by a cut-off function.  
\begin{lemma}[Extension from the extracellular part to $G$]
	\label{lm:extension-u^e}
  		For any $\omega$, there exists an extension operator $\Pi_\ve^\omega:H^1(G_\ve(\omega))\to H^1(G)$ 
		%of the function $\hat u^e_\ve$ from $G_\ve$ to $G$ 
		such that for any function $u(x)\in H^1(G_\ve)$, $u(x)=0$ on $S_0\cup S_1$, the following estimate holds:
		\begin{align*}
			%\label{expan est}
			\|\Pi_\ve^\omega u\|_{H^1(G)}  \le C \|\nabla  u\|_{L^2(G_\ve(\omega))}  .
		\end{align*}
	\end{lemma}
The proof can be found in \cite{pettersson2023bidomain}.
As a consequence of \eqref{apr_est_v}, Lemma \ref{lm:apr est aux} and Lemma \ref{lm:extension-u^e} we obtain \eqref{apr_est_ubound}.
%%%%%%%%%%%%%%%%%%%%%%%%%%%%%%%%%%%%%%%%%%%
\begin{lemma}[Quasi-extension from the intracellular part to $G$]
\label{lm:extension-u^i}
For $\hat u_\ve^i$  %\in H^1(F_\ve(\omega))$ 
satisfying \eqref{apr_est_u}, there exists $\tilde{u}_\ve^i$ defined in $G$ and such that
for any $t\in [0,T]$ and $\omega\in \Omega$, $\tilde{u}_\ve^i(t,\,\cdot \, , \omega)\in H^1(G)$, and 
\begin{itemize}
    \item[(i)] 
    $
    \displaystyle
    \int_{F_\ve(\omega)}|\tilde{u}_\ve^i - \hat u_\ve^i|^2\, dx
    + \ve \int_{M_\ve(\omega)}|\tilde{u}_\ve^i - \hat u_\ve^i|^2\, dS(x) \le C\, \ve^2$.
    \item[(ii)] For any $x_1\in [0,L]$, $\tilde{u}_\ve^i$ is constant on cross sections of axons $\nabla_{x'}\tilde{u}_\ve^i|_{A_\ve(\omega)} = 0$, and 
    \begin{align*}
    \ve^2 \int_G |\nabla_{x'} \tilde{u}_\ve^i|^2\, dx \le C.
    \end{align*}
\end{itemize}
\end{lemma}
\begin{proof}
We set 
\begin{align}
\tilde{u}_\ve^i(t,x,\omega)= \sum_{j\in J_\ve} \chi\big(|x'/\ve -x_j(\omega)|-r_j(\omega)\big)\frac{1}{|A_{\ve j}(\omega)|} \int_{A_{\ve j}({\omega})} \hat{u}_\ve^i(t,x_1, x', \omega)\, dx',
\label{tildeu}
\end{align}
where $\chi(r)$ is a smooth cut-off function such that $\chi(r)=1$ if $r\le0$, and $\chi(r)=0$ if $r>d_0/2$ ($d_0$ is the minimal distance between the axons, see \ref{H1}). % Since  is equal to the mean value of  on axons $A_\ve(\omega)$
Then, using the Poincaré inequality on cross sections of axons, we obtain the bound (i), and (ii) follows directly from the construction of the function $\tilde {u}_\ve^i$.
\end{proof}
 
%%%%%%%%%%%%%%%%%%%%%%%%%%%%%%%%%%%%%%%%%%%
	%\vspace{2mm}
	\subsection{Compactness and passage to the limit} 
	\label{sec:pass-to-lim}
{In view of a priori estimates \eqref{apr_est_v}, \eqref{apr_est_g}, up to a subsequence}, we have the following stochastic two-scale convergence on random surfaces $M_\ve(\omega)$: 
\begin{align*}
	%\label{conv_v}
	\hat v_\ve &\xrightharpoonup[]{2s} \ \hat v_0\in L^4(G_T; L^4(\Omega, \mu)),\\
	%\label{conv_g}
	\hat g_\ve &\xrightharpoonup[]{2s} \ \hat g_0\in L^2(G_T; L^2(\Omega, \mu)),\\
	%\label{conv_dtv}
	\partial_t \hat v_\ve &\xrightharpoonup[]{2s} \ \partial_t \hat v_0\in L^2(G_T; L^2(\Omega, \mu)),\\
	%\label{conv_dtg}
	\partial_t \hat g_\ve &\xrightharpoonup[]{2s}\ \partial_t \hat g_0\in L^2(G_T; L^2(\Omega, \mu)),
\end{align*}
as $\ve \to 0$. Moreover, for any fixed $t\in [0,T]$,
\begin{align}
	\label{conv_vT}
	\hat v_\ve &\xrightharpoonup[]{2s} \ \hat v_0
	\ \text{in}\  L^2(G; L^2(\Omega, \mu)),\\
	\label{conv_gT}
	\hat g_\ve &\xrightharpoonup[]{2s} \ \hat g_0
	\ \text{in}\  L^2(G; L^2(\Omega, \mu)).
\end{align}
By \eqref{apr_est_u}, Lemmas \ref{lm:extension-u^e}, \ref{lm:extension-u^i}, and Theorem 3.7 in \cite{bourgeat1994stochastic},
possibly passing to a further subsequence, we have the following convergence:
\begin{align}
\label{conv_Pue}
\Pi_\ve^\omega\hat u^e_\ve(t,x,\omega) & \xrightharpoonup[]{2s}  \ \hat u_0^e(t,x)
&\ \text{in}\ L^2(G_T; L^2(\Omega, {\mathcal P})),\\
\label{conv_gradPue}
\nabla_x (\Pi_\ve^\omega \hat u^e_\ve(t,x,\omega)) &\xrightharpoonup[]{2s} \ \nabla\hat u_0^e(t,x) + \hat z(t,x,\omega)
&\ \text{in}\ (L^2(G_T; L^2(\Omega, {\mathcal P})))^3,\\
&\text{where } {\rm curl}_\omega \hat z=0, \ \int_\Omega \hat z\, d{\mathcal P}(\omega)=0,&\nonumber
\\
\label{eq:conv-extended-u^i}
\tilde{u}_\ve^i(t,x,\omega) &\xrightharpoonup[]{2s}  \hat{u}_0^i(t,x,\omega) 
&\ \text{in} \ L^2(G_T; L^2(\Omega, {\mathcal P})),\\
\label{eq:conv-grad-extended-u^i}
\ve \nabla_{x'} \tilde{u}_\ve^i(t,x,\omega) &\xrightharpoonup[]{2s} \nabla_\omega \hat{u}_0^i(t,x,\omega) 
&\ \text{in} \ L^2(G_T; L^2(\Omega, {\mathcal P})).
\end{align}
%To obtain \eqref{eq:conv-grad-extended-u^i} we have also used Theorem 3.7 in \cite{wright1994stochastic}. 
Then
\begin{align}
\label{conv_uiG}
	\chi_{_{F_\ve}}(x) \hat u^{i}_\ve(t,x,\omega) &\xrightharpoonup[]{2s} \ \chi_{\A }(\omega)\hat u_0^{i}(t,x,\omega)
	\ \text{in}\ L^2(G_T; L^2(\Omega, {\mathcal P})),\\
	\label{conv_duiG}
	\chi_{_{F_\ve}}(x)\partial_{x_1}\hat u^{i}_\ve(t,x,\omega) &\xrightharpoonup[]{2s} \ \chi_{\A }(\omega) \partial_{x_1}\hat u_0^{i}(t,x,\omega)
	\ \text{in} \ L^2(G_T; L^2(\Omega, {\mathcal P})), \\ 
 & \hspace{0.8cm}\chi_{\A }(\omega)\nabla_\omega \hat{u}_0^i(t,x,\omega)=0.
\label{eq:grad-omega=0}
 \end{align}
Note that \eqref{eq:grad-omega=0} follows from \eqref{eq:conv-grad-extended-u^i} and the the fact that $\nabla_{x'}\tilde{u}_\ve^i=0$ for $x'\in A_{\ve}(\omega)$.
{In the next lemma we prove the convergence of traces of $\hat u_\ve^e, \hat u_\ve^i$ on $M_\ve(\omega)$.}
\begin{lemma}
	\label{lm:conv_u_ds}
	For the traces of $\hat u^e_\ve, \hat u^{i}_\ve$ on the random surface $M_\ve(\omega)$ the following convergences holds:
	\begin{align}
		\label{conv_ue_ds}
		\hat u^e_\ve(t,x,\omega) &\xrightharpoonup[]{2s} 
  %\xrightarrow[]{2s}
  \ \hat u_0^e(t,x)
		\ \text{in}
		\ L^4(G_T; L^4(\Omega, \mu)),\\
		\label{conv_ui_ds}
		\hat u^{i}_\ve(t,x,\omega) &\xrightharpoonup[]{2s} \ \hat u_0^{i}(t,x,\omega)
		\ \mbox{in}\ L^4(G_T; L^4(\Omega, \mu)),
	\end{align}
	where $\hat u_0^e(t,x)$ and $\hat u_0^{i}(t,x,\omega)$ are given by \eqref{conv_Pue} and \eqref{eq:conv-extended-u^i}, respectively.
\end{lemma}
\begin{proof} 
% \ip{
% Since the limit function $\hat u_0^e(t,x)$ does not depend on $\omega$, its trace on $M_\ve(\omega)$ is well-defined. In view of the convergence \eqref{conv_Pue},
% \begin{align*}
%     \ve^4 \int_0^T \int_{G_\ve(\omega)} |\nabla \hat u^e_\ve - \nabla \hat u_0^e|^2\, dx\, dt \to 0, \quad \ve \to 0.
% \end{align*}
% }
%%%%%%%%%%%%%%%%%%%%%%%%%%%%%%%%%%%%%%%%%%%%%%%
We prove the convergence for the intracellular component  \eqref{conv_ui_ds}. The convergence of $\hat{u}_\ve^e$ \eqref{conv_ue_ds} is proved similarly. 
% \begin{align}
% &\ve \int_\Omega\int_0^T \int_{M_\ve(\omega)} |\Pi_\ve^\omega u_\ve^e - u_0^e|^4\, dS(x)dt d{\mathcal P}(\omega)
% \\
% &\le C
% \left( \ve \int_\Omega \int_0^T \|\Pi_\ve^\omega u_\ve^e - u_0^e\|_{H^1(G)}^4\, dt d{\mathcal P}(\omega)
% + \int_\Omega \int_0^T \|\Pi_\ve^\omega u_\ve^e - u_0^e\|_{L^4(G)}^4\right)\, dt d{\mathcal P}(\omega).
% \end{align}

%%%%%%%%%%%%%%%%%%%%%%%%
To this end, it is sufficient to show that, as $\ve\to 0$,
\begin{align}
{\cal I}(\ve) = \ve \int\limits_0^T \int\limits_{\Omega} \int\limits_{M_\ve(\omega)}
\tilde u^i_\ve (t,x,\omega) \varphi(t,x) \phi(T_{\frac{x'}{\ve}}\omega)
dS(x) d{\mathcal P}(\omega) dt \nonumber \\ 
- \int\limits_{G_T} \int\limits_{\Omega} 
\hat u^i_0 (t,x,\omega) \varphi(t,x) \phi(\omega)
d\mu(\omega) dxdt \to 0,
\label{44}
\end{align}
for any $\varphi\in C_0^\infty(G_T), \phi\in D^\infty(\Omega)$, where $D^\infty(\Omega)$ is defined in  \eqref{def:D^inft}, $\tilde u^i_\ve (t,x,\omega)$ is the quasi-extension defined by \eqref{tildeu}, 
and $\hat u_0^i(t,x,\omega)$ is the limit function in \eqref{eq:conv-extended-u^i}. We will follow the proof of Lemma 8.1 in \cite{piatnitski2020homogenization}. The main trick is to approximate the surface measure on $M_\ve(\omega)$ by the weighted Lebesgue measure $\rho_\delta(T_{x'/\ve}\omega)\, dx_1 dx'$ on $\mathbb{R}^3$ where
%Define 
\begin{align}
\rho_\delta(\omega)=\frac{1}{\delta^2}\int\limits_{\partial A(\omega)} 
k\left(\frac{y'}{\delta}\right)ds(y'),
\end{align}
$k(y)\in C_0^\infty(\mathbb{R}^2)$ being a non-negative symmetric function such that $\int_{\mathbb{R}^2} k(y)dy=1$, $\delta>0$.

We write ${\cal I}(\ve)$ as a sum
\begin{align*}
{\cal I}(\ve) &= {\cal I}_1(\ve, \delta) + {\cal I}_2(\ve, \delta) + {\cal I}_3(\delta),\\
{\cal I}_1(\ve, \delta)&=
\int\limits_{G_T} \int\limits_{\Omega}
\tilde u^i_\ve (t,x,\omega) \varphi(t,x) \phi(T_{\frac{x'}{\ve}}\omega)
\rho_\delta(T_{\frac{x'}{\ve}}\omega) \,d{\mathcal P}(\omega) dx_1 dx' dt \nonumber \\ 
&- \int\limits_{G_T} \int\limits_{\Omega} 
\hat u^i_0 (t,x,\omega) \varphi(t,x) \phi(\omega)\,
d\mu_\delta(\omega) \,dxdt;\\
{\cal I}_2(\ve,\delta)&=\ve \int\limits_0^T \int\limits_{\Omega} \int\limits_{M_\ve(\omega)}
\tilde u^i_\ve (t,x,\omega) \varphi(t,x) \phi(T_{\frac{x'}{\ve}}\omega)\,
dS(x) d{\mathcal P}(\omega) dt \nonumber \\
&
-
\int\limits_{G_T} \int\limits_{\Omega}
\tilde u^i_\ve (t,x,\omega) \varphi(t,x) \phi(T_{\frac{x'}{\ve}}\omega)
\rho_\delta(T_{\frac{x'}{\ve}}\omega) d{\mathcal P}(\omega)\, dx_1 dx' dt;\\
{\cal I}_3(\delta)&=\int\limits_{G_T} \int\limits_{\Omega} \hat u^i_0 (t,x, \omega) \varphi(t,x) \phi(\omega) \left(d{\mu_\delta}(\omega)-d\mu(\omega)\right)\,
dx dt. 
%\label{I2}
\end{align*}
Here $d\mu_\delta(\omega)=\rho_\delta(\omega) d{\mathcal P}(\omega)$ is the Palm measure of $\rho_\delta(T_{x'/\ve}\omega)\,dx_1 dx'$.

Thanks to \eqref{eq:conv-extended-u^i}
\begin{align}
% \int\limits_{G_T} \int\limits_{\Omega}
% \tilde u^i_\ve (t,x,\omega) \varphi(t,x) \phi(T_{\frac{x'}{\ve}}\omega)
% \rho_\delta(T_{\frac{x'}{\ve}}\omega) d{\mathcal P}(\omega) dx_1 dx' dt \nonumber \\ 
% \to \int\limits_{G_T} \int\limits_{\Omega} 
% \hat u^i_0 (t,x,\omega) \varphi(t,x) \phi(\omega)
% d\mu_\delta(\omega) dxdt
{\cal I}_1(\ve, \delta) \to 0, \ \ \ve\to 0.
\label{46}
\end{align}
Next, we estimate ${\cal I}_2(\ve, \delta)$. We fix $t, x_1, \omega$ and consider functions 
\begin{align}
\tilde w_\ve (x')=\tilde u^i_\ve (t,x_1,x',\omega)\varphi(t,x_1,x') \phi(T_{\frac{x'}{\ve}}\omega)
\label{tilde_w}
\end{align} 
extended by zero on  $\mathbb{R}^2$. Changing the variables we have
\begin{align}
&\int\limits_{\mathbb{R}^2} \tilde w_\ve (x') \rho_\delta(T_{\frac{x'}{\ve}}\omega) dx'=
\ve^2\int\limits_{\mathbb{R}^2} \tilde w_\ve (\ve y') \rho_\delta(T_{y'}\omega) dy'\nonumber\\
=& \frac{\ve^2}{\delta^2}\int\limits_{\mathbb{R}^2} \tilde w_\ve (\ve y') \int\limits_{\partial A(T_{y'}\omega)} 
k\left(\frac{y''}{\delta}\right)ds(y'') dy'\nonumber\\
=&
\frac{\ve^2}{\delta^2}\int\limits_{\mathbb{R}^2} \tilde w_\ve (\ve y') \int\limits_{\partial A(\omega)-y'} 
k\left(\frac{y''}{\delta}\right)ds(y'') dy'\nonumber\\
=&\frac{\ve^2}{\delta^2}\int\limits_{\mathbb{R}^2} \tilde w_\ve (\ve y') \int\limits_{\partial A(\omega)} 
k\left(\frac{z'-y'}{\delta}\right)ds(z') dy'\nonumber\\
% =&
% \frac{\ve^2}{\delta^2}\int\limits_{\partial A(\omega)} \int\limits_{\mathbb{R}^2}   
% k\left(\frac{z'-y'}{\delta}\right)\tilde w_\ve (\ve y') dy'ds(z')\nonumber\\
=&\frac{\ve^2}{\delta^2}\int\limits_{\partial A(\omega)} \int\limits_{\mathbb{R}^2}   
k\left(\frac{y'-z'}{\delta}\right)\tilde w_\ve (\ve z') dz'ds(y') \nonumber
\end{align}
and then 
\begin{align}
&\ve  \int\limits_{\ve \partial A(\omega)}
\tilde w_\ve (x') ds(x') - 
\int\limits_{\mathbb{R}^2} \tilde w_\ve (x') \rho_\delta(T_{\frac{x'}{\ve}}\omega) dx =%\nonumber\\=&
\ve^2 \int\limits_{\partial A(\omega)}
w_{\ve,\delta} (y') ds(y'),
\label{enter_w_ed}
\end{align}
with
\begin{align}
w_{\ve,\delta} (y')= \tilde w_\ve (\ve y') - 
\frac{1}{\delta^2}\int\limits_{\mathbb{R}^2} k\left(\frac{y'-z'}{\delta}\right)\tilde w_\ve (\ve z') dz'.
\end{align}
Due to the properties of mollifiers (see, e.g., Lemma 3.5 in \cite{Majda_Bertozzi_2001}), we have
\begin{align}
\|w_{\ve,\delta} \|_{H^1({\mathbb{R}^2})}
\leq C \|\tilde w_\ve (\ve \ \cdot \ )\|_{H^1({\mathbb{R}^2})}, \quad
\|w_{\ve,\delta} \|_{L^2({\mathbb{R}^2})}
\leq C \delta\|\tilde w_\ve (\ve\ \cdot \ )\|_{H^1({\mathbb{R}^2})}. 
\label{est_w_ve_delta}
\end{align}
Using the Cauchy-Schwarz inequality, trace inequality, and \eqref{est_w_ve_delta} we get
\begin{align}
&\ve^2 \int\limits_{\partial A(\omega)}
|w_{\ve,\delta} (y')| ds(y')
\leq C\ve \Big(\sum_j \int\limits_{\partial A_j(\omega)}
|w_{\ve,\delta} (y')|^2 ds(y')\Big)^{1/2}\nonumber\\
&\leq C\ve \Big(\sum_j \|w_{\ve,\delta} \|_{H^1({A_j(\omega)})} \|w_{\ve,\delta} \|_{L^2(A_j(\omega))}\Big)^{1/2}\nonumber\\
&\leq C\ve \Big( \delta\|w_{\ve,\delta} \|^2_{H^1({\mathbb{R}^2})} + \delta^{-1}\|w_{\ve,\delta} \|^2_{L^2({\mathbb{R}^2})}\Big)^{1/2}\nonumber\\
&\leq C \ve \sqrt{\delta} \, \|\tilde w_\ve (\ve \ \cdot \ )\|_{H^1({\mathbb{R}^2})}=
C\sqrt{\delta}
\left(\|\tilde w_\ve \|_{L^2({\mathbb{R}^2})}+
\ve\|\nabla\tilde w_\ve \|_{L^2({\mathbb{R}^2})}\right).
\label{est_w_ve_delta_next}
\end{align}
Integrating \eqref{enter_w_ed} in $t, x_1, \omega$, using \eqref{est_w_ve_delta_next} 
 and Lemmas \ref{lm:apriori-est}, \ref{lm:extension-u^i} we obtain
\begin{align}
&\label{I1}\bigl| {\cal I}_2(\ve,\delta) \bigr|\leq C\sqrt{\delta}.
% ,\\
% {\cal I}_1(\ve,\delta)&=\ve \int\limits_0^T \int\limits_{\Omega} \int\limits_{M_\ve(\omega)}
% \tilde u^i_\ve (t,x,\omega) \varphi(t,x) \phi(T_{\frac{x'}{\ve}}\omega)
% ds(x') d{\mathcal P}(\omega) dt \nonumber \\
% &-
%  \int\limits_{G_T} \int\limits_{\Omega}
% \tilde u^i_\ve (t,x,\omega) \varphi(t,x) \phi(T_{\frac{x'}{\ve}}\omega)
% \rho_\delta(T_{\frac{x'}{\ve}}\omega) d{\mathcal P}(\omega) dx_1 dx' dt.
\end{align}
To estimate ${\cal I}_3(\delta)$, we assume for simplicity that the function $\hat u^i_0 (t,x_1,x', \omega)$
is continuously differentiable in $x'$ (otherwise, it will be approximated by continuously differentiable in $x'$ functions). Applying Campbell's formula \eqref{Campbell} 
 to the function 
$\hat u^i_0 (t,x_1,x', T_{\frac{x'}{\ve}}\omega)\varphi(t,x_1,x') \phi(T_{\frac{x'}{\ve}}\omega)$
we obtain
\begin{align}
&\int\limits_{G_T} \int\limits_{\Omega}
\hat u^i_0(t,x,\omega) \varphi(t,x) \phi(\omega)
d\mu(\omega) dx dt \nonumber \\ 
=&\ve \int\limits_0^T \int\limits_0^L  \int\limits_{\Omega} 
\int\limits_{\ve \partial A(\omega)}
\hat u^i_0 (t,x_1,x',T_{\frac{x'}{\ve}}\omega) \varphi(t,x) \phi(T_{\frac{x'}{\ve}}\omega)
ds(x') d{\mathcal P}(\omega) dx_1 dt
\label{52}
\end{align}
and
\begin{align}
&\int\limits_{G_T} \int\limits_{\Omega} 
\hat u^i_0 (t,x,\omega) \varphi(t,x) \phi(\omega)
d\mu_\delta(\omega) dxdt  \nonumber \\ 
=&\ve^2 \int\limits_{G_T}  \int\limits_{\Omega}  
\hat u^i_0 (t, x, T_{\frac{x'}{\ve}}\omega) \varphi(t,x) \phi(T_{\frac{x'}{\ve}}\omega)
\rho_\delta(T_{\frac{x'}{\ve}}\omega) d{\mathcal P}(\omega) dx dt
\label{53}
\end{align}
for any $\ve>0$.
Now we fix $t, x_1, \omega$ and consider the function 
$$
\tilde z_\ve (x')=\tilde u_0^i (t,x_1,x',T_{\frac{x'}{\ve}}\omega)\varphi(t,x_1,x') \phi(T_{\frac{x'}{\ve}}\omega)
$$ 
extended by zero on  $\mathbb{R}^2$ in place of \eqref{tilde_w}.
\begin{comment}
Then, by changing the variables, we get
\begin{align}
&\ve \int \limits_{\ve \partial A(\omega)} \tilde z_\ve (x')\, ds(x') - \ve^2 \int \limits_{\mathbb R^2} \tilde z_\ve (x') \rho_\delta(T_{\frac{x'}{\ve}}\omega) \, dx' \nonumber\\
&=\ve^2 \int \limits_{\partial A(\omega)} \Big[ \tilde z_\ve (\ve y') - \frac{1}{\delta^2}  \int\limits_{\mathbb{R}^2}   
k\left(\frac{y'-\eta'}{\delta}\right)\tilde z_\ve (\ve \eta') d\eta'\Big]\, ds(y')\label{enter_z_ed}\\
&=: \ve^2 \int \limits_{\partial A(\omega)} \tilde z_{\ve, \delta}(y')\, ds(y').
\end{align}
Now we proceed as before and get the estimate similar to \eqref{est_w_ve_delta_next}:
\begin{align}
\label{est_z_ve_delta_next}
\ve^2 \int \limits_{\partial A(\omega)} | \tilde z_{\ve, \delta}(y')|\, ds(y')
\le C\ve \sqrt{\delta}\|\tilde z_\ve(\ve \cdot)\|_{H^1(\mathbb R^2)}.
\end{align}
\end{comment}
Then we proceed as before for ${\cal I}_2(\ve, \delta)$ and get the following estimate for ${\cal I}_3(\delta)$:
\begin{align}
&\bigl| {\cal I}_3(\delta) \bigr| \leq C\sqrt{\delta}
\Bigl( \int\limits_{G_T} \int\limits_{\Omega} 
\left( 
\bigl|\hat u^i_0 \bigr|^2
+\bigl| \nabla_{\omega}\hat u^i_0   \bigr|^2    
+\ve^2\bigl| \nabla_{x'} \hat u^i_0    \bigr|^2
\right)d{\mathcal P}(\omega) dx dt \Bigr)^{1/2}. \label{I2estim}
\end{align}
%Note that there is one difference between \eqref{est_w_ve_delta_next} and \eqref{est_z_ve_delta_next}: in the first one we have $\tilde u_\ve^i(t,x,\omega)$ and in the second one $\hat u_0^i(t,x,T_{x'/\ve}\omega)$, which gives the stochastic gradient $\nabla_\omega \hat u_0^i$ in \eqref{I2estim}. 

Taking into account the convergence \eqref{46} and estimates \eqref{I1}, \eqref{I2estim}
%for left-hand sides  and right-hand sides of \eqref{44} and \eqref{46},
we pass to the limit as $\ve, \delta \to 0$ and obtain \eqref{44}. 

In the case when the function $\hat u^i_0 (t,x_1,x', \omega)$
does not have $C^1$- regularity in $x'$, one can % has to 
approximate it by %continuously differentiable in $x'$ functions.
\begin{align}
\hat u_{0,\delta} (t,x_1,x', \omega) = 
\frac{1}{\delta^2}\int\limits_{S} k\left(\frac{x'-y'}{\delta}\right)\hat u^i_0 (t,x_1,y', \omega) dy',
\end{align}
and use \eqref{I2estim} for sufficiently small $\ve>0$.
Lemma \ref{lm:conv_u_ds} is proved.
\end{proof}

To pass to the limit in
\eqref{eq:orig-prob-nonlocal-hat}, we use the Minty method \cite{minty1962}.
For an arbitrary function $\psi(t,x,\omega)\in C^2(G_T, L^2(\Omega,\mu))$, we set $\psi_\ve(t,x,\omega)=\psi(t,x, T_{x'/\ve}\omega)$ which will be used as a test function for $g_\ve$. 
\noindent
We define also $\phi_\ve(t,x,\omega)=\begin{cases} \phi_\ve^{i}(t,x,\omega)\,\, \mbox{if}\,\, x\in F_\ve (\omega)\\\phi_\ve^{e}(t,x,\omega)\,\, \mbox{if}\,\, x\in G_\ve (\omega) \end{cases}$ which we will be chosen later as a test function for $v_\ve$. 

Let $\phi^e_0\in C^2(G_T)$ be an arbitrary function that equals zero on $S_0\cup S_L$. Take $\phi^e_\ve$ as a solution of the problem
\begin{align}
	\label{def phi^e_eq}
            &\di{\hat\sigma^e\left(t,|\nabla\phi^e_\ve|\right)\nabla \phi^e_\ve}
			=\frac{1}{1-\Lambda}\, \di{\hat\sigma_{\rm hom}^e\left( \nabla\phi^e_0\right) } \, &\text{in} \  &G_\ve(\omega),\\
			&\hat\sigma^e\left(t,|\nabla\phi^e_\ve|\right)\nabla\phi^e_\ve\cdot \nu   = 0 \, &\text{on} \ &M_\ve(\omega), \label{def phi^e_mc} \\
			&\hat\sigma^e\left(t,|\nabla\phi^e_\ve|\right)\nabla\phi^e_\ve \cdot \nu   = \hat\sigma_{\rm hom}^e\left(\nabla\phi^e_0\right) \cdot \nu \, & \text{on}\ &\Sigma,  \label{def phi^e_bc1}\\
			&\phi^e_\ve = 0 \, &\text{on}\  &S_{0} \cup S_L. 
            \label{def phi^e_bc2} 
\end{align}
Recall that $\Lambda={\mathcal P}(\A )$ (see \eqref{dens}), the transformed effective extracellular conductivity $\hat \sigma^e_{\rm hom}$ is define by \eqref{pot}, and
\begin{align}
\label{hat eff tensor}
&\hat \sigma^e_{\rm hom}(t, \xi)=e^{-\lambda t}\nabla\Phi(e^{\lambda t}\xi), \quad \xi\in \mathbb{R}^3.
\end{align}
				
\begin{lemma}
\label{lm:conv phi^e}
For the unique solution $\phi^e_\ve$ of \eqref{def phi^e_eq}--\eqref{def phi^e_bc2}, the following uniform in $t\in[0,T]$ convergence holds for a.a. $\omega\in \Omega$:
\begin{align}
\label{conv phi^e}
\Pi_\ve^\omega\phi^e_\ve \rightarrow \phi^e_0, \ \mbox{strongly in}\  L^4(G),
\end{align}
as $\ve\to 0$. Moreover,
\begin{align}
\label{conv_grad_phie}
\nabla(\Pi^\omega_\ve\phi^e_\ve) \xrightharpoonup[]{2s} \nabla\phi_0^e + z(t,x,\omega)
\ \ \text{in}\ (L^2(G_T; L^2(\Omega, {\mathcal P})))^3,
%& {\rm curl}_\omega z=0, \ \int_\Omega z\, d{\mathcal P}(\omega)=0,&\nonumber
			\end{align}
where ${\rm curl}_\omega z=0$, $\displaystyle\int\limits_\Omega z\, d{\mathcal P}(\omega)=0$.
\end{lemma}

\begin{proof} The proof of \eqref{conv_grad_phie} can be done in the standard way (for example, as Lemma 4 in \cite{pettersson2023bidomain}) using additionally the monotonicity of 
the vector field $\hat \sigma^e(t,|\xi|)\xi$ (see also Lemma \ref{lm:monotonicity}).
%Due to \cite{Zhikov} there exits a unique solution such that, as $\ve\to 0$:
\end{proof}	
\begin{corollary}
\label{cor:conv phi^e bound}
For the solution $\phi^e_\ve$ of \eqref{def phi^e_eq}--\eqref{def phi^e_bc2} the following strong  stochastic two-scale convergence on $M_\ve$ holds (see \eqref{def_2s_on_ax}), as $\ve\to 0$:
\begin{align}
	\label{conv phi^e bound}
	\Pi_\ve^\omega\phi^e_\ve\xrightarrow[]{2s}\phi^e_0 \ \  \mbox{strongly in}\ L^4(G_T; L^4(\Omega, \mu)).
\end{align}
%\begin{align}
%		\label{conv phi_i}
%	\end{align}
\end{corollary}
\begin{proof}
We get this result applying Lemma \ref{lm:apr est aux} with $u=\Pi_\ve^\omega\phi^e_\ve-\phi^e_0$ and using the compactness of the embedding $H^1(G)\subset L^4(G)$. 
\end{proof}

Next, let us proceed with the definition of the test function $\phi^i_\ve$. Taking an arbitrary bounded function $\phi^{i}_0(t,x,\omega)\in C^2(G_T, {\cal K(\A )})$ that equals zero on $S_0\cup S_L$, we define  test functions 
$\phi^{i}_\ve$ on each axon 
%\begin{align*}
%\label{axon}
%F^k_{\ve j}= 
%\biggl\{x_1\in [0,L], (x_2,x_3)=x'\in \mathbb{R}^2: \left| x'-%x'^k_{\ve j} \right|<\ve r_k \biggr\}, \ \ k=1,2 
%\end{align*}
by setting
\begin{align}
\label{def phi^i}
\phi^{i}_\ve(t,x,\omega)=\phi^{i}_0(t, x_1, \ve x'_{j}(\omega) , T_{\frac{x'}{\ve}}\omega), \quad t\in [0,T], \ x\in 	F_{\ve j}(\omega).  
\end{align}
%the functions $	\phi^i_\ve$ and $\phi^i_0$ don't differ much, so 
By construction \eqref{def phi^i} of $\phi^{i}_\ve$, we have the  strong  stochastic two-scale convergence (see \eqref{def_2s_on_ax}), as $\ve\to 0$:
\begin{align}
%\chi_{_{G_\ve}}
\phi^{i}_\ve \xrightarrow[]{2s}  \phi^{i}_0\ \  \mbox{strongly in}\   L^4(G_T; L^4(\Omega, {\mathcal P})). 
\end{align}
For brevity, in what follows, we omit arguments in the conductivity function and write $\hat \sigma_\ve(|\xi|)$ for $\hat \sigma_\ve(t,x, \omega, |\xi|)$. Thanks to the assumption \ref{H5},  it holds that
\begin{align}
		\label{equiv-monotonicity}
		(\hat\sigma_\ve(|\xi|)\xi - \hat\sigma_\ve(|\eta|)\eta)\cdot (\xi-\eta)\ge \underline{\sigma} |\xi-\eta|^2, \quad \forall \xi, \eta \in \mathbf R^3,
\end{align}
with the positive constant $\underline{\sigma}$ (see \ref{lm:monotonicity}). Using \eqref{equiv-monotonicity} for $\nabla \phi_\ve, \nabla \hat{u}_\ve$  we get  
\begin{align}
\label{ineq1}
&\int\limits_{F_\ve(\omega) \cup G_\ve(\omega)}(\hat\sigma_\ve(|\nabla \phi_\ve|)\nabla \phi_\ve - \hat\sigma_\ve(|\nabla \hat u_\ve|)\nabla \hat u_\ve)\cdot (\nabla \phi_\ve-\nabla \hat u_\ve) dx\notag\\
&=\int\limits_{F_\ve(\omega) \cup G_\ve(\omega)}\hat\sigma_\ve(|\nabla \phi_\ve|)\nabla \phi_\ve \cdot (\nabla \phi_\ve-\nabla \hat u_\ve) dx \notag\\
&\quad\quad \quad- (\hat{\mathcal L}_\ve(\hat v_\ve), [\phi_\ve-\hat u_\ve])_{L^2(M_\ve(\omega))} 	- \int \limits_{\Sigma} e^{-\lambda t}J_\ve^e (\phi^e_\ve-\hat u^e_\ve) \, dS(x)\\
&\ge \int\limits_{F_\ve(\omega) \cup G_\ve(\omega)}\underline{\sigma} |\nabla \phi_\ve-\nabla \hat u_\ve|^2dx.	\notag
\end{align}
Then \eqref{ineq1} implies that for a.a. $\omega\in\Omega$ and all $t\in[0,T]$, 
\begin{align}
&-\ve	\int\limits_0^T \int\limits_{M_\ve(\omega)}
\left( c_m \partial_t \hat v_\ve + \hat I_{\rm ion}(\hat v_\ve, \hat g_\ve) \right)
\left([\phi_\ve]-\hat v_\ve \right)dS(x)dt\notag\\
&\leq
\int\limits_0^T\int\limits_{F_\ve(\omega) \cup G_\ve(\omega)}\hat\sigma_\ve(|\nabla \phi_\ve|)\nabla \phi_\ve \cdot (\nabla \phi_\ve-\nabla \hat u_\ve) dxdt
 \notag\\
&-	\int \limits_{\Sigma} e^{-\lambda t}J_\ve^e (\phi^e_\ve-\hat u^e_\ve) \, dS(x)-\int\limits_0^T \int\limits_{F_\ve(\omega) \cup G_\ve(\omega)}\underline{\sigma} |\nabla \phi_\ve-\nabla \hat u_\ve|^2dxdt.
\label{70}
\end{align}

\begin{comment}
\begin{align}
	&\ve	\int\limits_0^T \int\limits_{M_\ve}
	\hat I_{\rm ion}(\hat v_\ve, \hat g_\ve) 
	\left(\hat v_\ve- [\phi_\ve]\right)dS(x)dt\notag\\
	&=\ve	\int\limits_0^T \int\limits_{M_\ve}
	\left(\hat I_{\rm ion}(\hat v_\ve, \hat g_\ve) - \hat I_{\rm ion}([\phi_\ve], \psi_\ve) \right)
	\left(\hat v_\ve- [\phi_\ve] \right)dS(x)dt\notag\\
	&+\ve	\int\limits_0^T \int\limits_{M_\ve}
	\left( \hat I_{\rm ion}([\phi_\ve], \psi_\ve) \right)
	\left(\hat v_\ve- [\phi_\ve] \right)dS(x)dt,
\end{align}
where $\psi\in C^\infty([0,T]\times M_\ve)$.
\end{comment}
We write $\hat I_{\rm ion}(\hat v_\ve, \hat g_\ve)=\hat I_{\rm ion}(\hat v_\ve, \hat g_\ve)+ \hat I_{\rm ion}([\phi_\ve], \psi_\ve)-\hat I_{\rm ion}([\phi_\ve], \psi_\ve)$ and thanks to \eqref{lambda} obtain
\begin{align}
&\int\limits_0^T \int\limits_{M_\ve(\omega)}
	\left(\hat I_{\rm ion}(\hat v_\ve, \hat g_\ve) - \hat I_{\rm ion}([\phi_\ve], \psi_\ve) \right)
	\left(\hat v_\ve- [\phi_\ve] \right)dS(x)dt\notag\\
	&+ \int\limits_0^T \int\limits_{M_\ve(\omega)}
		\left( (b+\lambda) \hat g_\ve - \theta \hat v_\ve
-(b+\lambda) \psi_\ve + \theta [\phi_\ve]\right)\left(\hat g_\ve- \psi_\ve\right) dS(x)dt
\geq 0.
\label{72}
\end{align}
Multiplying the equation for $\hat{g}_\ve$ in \eqref{eq:orig-prob-nonlocal-hat} by $\ve \left(\hat g_\ve- \psi_\ve\right)$, 
replacing   $\hat v_{\ve}=\hat v_{\ve}+[\phi_\ve]-[\phi_\ve]$, $\hat g_{\ve}=\hat g_{\ve}+\psi_\ve-\psi_\ve$ in the right-hand side and integrating, we obtain
\begin{align}
	&\frac{\ve}{2} \int\limits_{M_\ve(\omega)}\hat g^2_\ve(T) \, dS(x)
	-\frac{\ve}{2} \int\limits_{M_\ve(\omega)} g^2_\ve(0) \, dS(x)-
	\ve\int\limits_0^T\int\limits_{M_\ve(\omega)} \partial_t\hat g_\ve  \, \psi_\ve\, dS(x)\, dt\nonumber\\
&  	+\ve\int\limits_0^T\int\limits_{M_\ve(\omega)}	\bigg(
	\theta [\phi_\ve] + a e^{-\lambda t}-(\lambda+b) \psi_\ve\bigg)
	(\psi_\ve-\hat g_\ve)   \, dS(x)\, dt	\label{73} \\   
&+	\ve\int\limits_0^T\int\limits_{M_\ve(\omega)}
	\bigg(
	\theta (\hat v_\ve -[\phi_\ve]) -(\lambda+b) (\hat g_\ve -\psi_\ve)\bigg)
	(\psi_\ve-\hat g_\ve)   \, dS(x)\, dt =0.
	\notag
\end{align}
Then \eqref{70}, \eqref{72} yield %\eqref{73}	
\begin{align}
	&\frac{\ve c_m}{2} \int\limits_{M_\ve(\omega)}\hat v^2_\ve(T) \, dS(x)
	+\frac{\ve}{2} \int\limits_{M_\ve(\omega)}\hat g^2_\ve(T) \, dS(x)
\nonumber\\
	-&\frac{\ve c_m}{2} \int\limits_{M_\ve(\omega)} v^2_\ve(0) \, dS(x)
	-\frac{\ve}{2} \int\limits_{M_\ve(\omega)} g^2_\ve(0) \, dS(x)\nonumber\\
	-&\ve c_m\int\limits_0^T\int\limits_{M_\ve} \partial_t\hat v_\ve  [\phi_\ve]\, dS(x)\, dt
	-\ve \int\limits_0^T\int\limits_{M_\ve(\omega)} \partial_t\hat g_\ve  \psi_\ve\, dS(x)\, dt\nonumber\\
	+&  \ve	\int\limits_0^T \int\limits_{M_\ve(\omega)}
	\left( \hat I_{\rm ion}([\phi_\ve], \psi_\ve) \right)
	\left(\hat v_\ve- [\phi_\ve] \right)dS(x)dt	\nonumber \\   
	+&\ve\int\limits_0^T\int\limits_{M_\ve(\omega)}
	\bigg(
	\theta [\phi_\ve] +ae^{-\lambda t}-(\lambda+b)\psi_\ve\bigg)
	(\psi_\ve-\hat g_\ve)   \, dS(x)\, dt \nonumber \\   
	\leq&	
	\int\limits_0^T\int\limits_{F_\ve(\omega) \cup G_\ve(\omega)}\hat\sigma_\ve(|\nabla \phi_\ve|)\nabla \phi_\ve \cdot (\nabla \phi_\ve-\nabla \hat u_\ve) dxdt \nonumber \\   
	-&\int\limits_0^T\int \limits_{\Sigma} e^{-\lambda t}J_\ve^e (\phi^e_\ve-\hat u^e_\ve) \, dS(x)dt-\int\limits_0^T \int\limits_{F_\ve(\omega) \cup G_\ve(\omega)}\underline{\sigma} |\nabla \phi_\ve-\nabla \hat u_\ve|^2dxdt.
	\label{almost_ready_forlimit}
\end{align}
Next, we integrate \eqref{almost_ready_forlimit} over $\Omega$ and pass to the limit as $\ve\to 0$. 

Due to the lower semi-continuity of $L^2$-norms of functions $\hat v_\ve, \hat g_\ve$ and taking into account \eqref{conv_vT}, \eqref{conv_gT}, we get:
\begin{align}
	& \liminf\limits_{\ve \to 0}\, 
	\ve	 \int\limits_{\Omega}\int\limits_{M_\ve(\omega)} \hat v^2_\ve(T) \, dx \, d{\mathcal P}(\omega)
	\ge\int\limits_{\Omega}\int\limits_{G} \hat v^2_0(T) \, dx\, d\mu(\omega),  \label{semi_v}	\\  
	& \liminf\limits_{\ve \to 0}\, 
	\ve	 \int\limits_{\Omega}\int\limits_{M_\ve(\omega)} \hat g^2_\ve (T)\, dx \, d{\mathcal P}(\omega)
	\ge\int\limits_{\Omega}\int\limits_{G} \hat g^2_0(T) \, dx\, d\mu(\omega). \label{semi_g}  
\end{align}
Further, we have the following convergence 
%considering \eqref{eq:weak phi^e} with $\zeta=\Pi_\ve\phi^e_\ve-\Pi_\ve \hat u^e_\ve$ and passing to the limit with a help of Lemma \ref{lm:conv phi^e}, 
\begin{align}
\label{lim_ue}
\int\limits_0^T&\int\limits_\Omega\int\limits_{G_\ve(\omega)}  \hat\sigma^e(|\nabla \phi^e_\ve|) \nabla \phi^e_\ve\cdot\nabla (\phi^e_\ve- \hat u^e_\ve) \, dx\, 	d{\mathcal P}(\omega) \, dt\notag\\ 	&\rightarrow 
\int\limits_0^T\int\limits_G\hat\sigma_{\rm hom}(\nabla \phi_0^e)\cdot(\nabla \phi_0^e-\nabla \hat u_0^e) dx\, dt.
\end{align}
Indeed, multiplying \eqref{def phi^e_eq} by $(\phi^e_\ve-\hat u^e_\ve)$, integrating by parts, and using boundary conditions \eqref{def phi^e_mc}-\eqref{def phi^e_bc2} we obtain
\begin{align}
&\int\limits_{G_\ve(\omega)}  \hat\sigma^e(|\nabla \phi^e_\ve|) \nabla \phi^e_\ve\cdot\nabla (\phi^e_\ve-\hat u^e_\ve) \, dx %\notag\\ 	
= \int\limits_G\hat\sigma_{\rm hom}(\nabla \phi_0^e)\cdot(\nabla\Pi^\omega_\ve \phi_\ve^e-\nabla \Pi^\omega_\ve \hat u_\ve^e) dx\,\notag\\ 
&+
\int\limits_G\di{\hat\sigma_{\rm hom}(\nabla \phi_0^e)}
\left(1-\frac{1}{1-\Lambda}\chi_{_{G_\ve}}(x)
\right)\left(\Pi^\omega_\ve\phi^e_\ve-\Pi^\omega_\ve\hat u^e_\ve\right)dx.
\label{84}
\end{align}
%The first term in \eqref{84} 
Due to \eqref{conv_gradPue}, \eqref{conv_grad_phie}
\begin{align}
 &\int\limits_0^T\int\limits_\Omega\int\limits_G\hat\sigma_{\rm hom}(\nabla \phi_0^e)\cdot(\nabla\Pi^\omega_\ve \phi_\ve^e-\nabla \Pi^\omega_\ve\hat u_\ve^e) dx\, d{\mathcal P}(\omega) \, dt\notag\\ 
&\quad\rightarrow \int\limits_0^T
\int\limits_G \hat\sigma_{\rm hom}(\nabla \phi_0^e)\cdot(\nabla \phi_0^e-\nabla \hat u_0^e) dx\, dt.
\label{85}
\end{align}
Since $\chi_{_{G_\ve}}(x)$ converges strongly stochastically two-scale to $\chi_{\Omega\setminus \bar\A }(\omega)$ and in view of \eqref{conv phi^e}, \eqref{conv_Pue} we also have 
\begin{align}
&\int\limits_0^T\int\limits_\Omega\int\limits_G\di{\hat\sigma_{\rm hom}(\nabla \phi_0^e)}
\left(1-\frac{1}{1-\Lambda}\chi_{_{G_\ve}}(x)
\right)\left(\Pi^\omega_\ve\phi^e_\ve-\Pi^\omega_\ve\hat u^e_\ve\right)dx d{\mathcal P}(\omega) \, dt\notag\\ 
&\rightarrow \int\limits_0^T
\int\limits_G \hat\sigma_{\rm hom}(\nabla \phi_0^e)\cdot(\phi_0^e - \hat u_0^e) dx\, dt\int\limits_\Omega\left(1-\frac{1}{1-\Lambda}\chi_{\Omega\setminus \bar\A }(\omega)\right)d{\mathcal P}(\omega)=0,
\label{86}
\end{align}
thus \eqref{84} is established.

By the construction of the test function $\phi_\ve^i$, using \eqref{conv_duiG} and Lemma \ref{lemma_radii} we conclude that
\begin{align}
%\label {lim_ui}
&\int\limits_0^T\int\limits_\Omega\int\limits_{F_\ve} 
\hat\sigma^i(|\nabla \phi^i_\ve|)\nabla\phi^{i}_\ve \cdot\nabla(\phi^{i}_\ve-\hat u^{i}_\ve) \, dx\, d{\mathcal P}(\omega) \, dt	\notag\\
&\quad\rightarrow \int\limits_0^T\int\limits_\A \int\limits_{G} \hat\sigma^i(|\nabla \phi^i_0|)\partial_{x_1}\phi^{i}_0  \,\partial_{x_1} (\phi^{i}_0-\hat u_0^{i}) \, dx\, d{\mathcal P}(\omega) \, dt	
%\end{align}
%\begin{comment}
%and due to Lemma \ref{lemma_radii} we get 
%\begin{align}
%\label {lim_ui}
%&\int\limits_0^T\int\limits_\A \int\limits_{G} \hat\sigma^i(|\nabla \phi^i_0|)\partial_{x_1}\phi^{i}_0  \,\partial_{x_1} (\phi^{i}_0-\hat u_0^{i}) ^\, dx\, d{\mathcal P}(\omega) \, dt 
\notag\\
&\quad=\frac{1}{2}\int\limits_0^T\int\limits_\Omega\int\limits_{G} r(\omega)\hat\sigma^i(|\nabla \phi^i_0|)\partial_{x_1}\phi^{i}_0  \,\partial_{x_1} (\phi^{i}_0-\hat u_0^{i}) \, dx\, d\mu(\omega) \, dt.	
\label{87}
\end{align}
%\end{comment}

%We also set $[\phi_0^k]=\phi^{ik}_0-\phi^{e}_0$ on $G$.
Finally, pass to the limit in the first term of the last line of \eqref{almost_ready_forlimit}. To this end, consider a harmonic function 
$\Psi_\ve(t,\cdot) \in H^1(G)$ being a solution of the following problem:
\begin{align}
&\Delta \Psi_\ve (t,x)=0 & \text{in } &G, \notag\\
&\nabla \Psi_\ve(t,x)\cdot \nu = e^{-\lambda t} J_\ve^e & \text{on } &\Sigma, \label{Psi}\\
&\Psi_\ve(t,x)=0 & \text{on } &S_0\cup S_L. \notag
\end{align}
Using \ref{H8} one can show that  
\begin{align}
\nabla \Psi_\ve \to \nabla \Psi
 \quad \text{strongly in } L^2(G_T),
 \label{gradPsi}
\end{align}
where $\Psi$ solves \eqref{Psi} with  $e^{-\lambda t} J_0^e$ in place of 
$e^{-\lambda t} J_\ve^e$. %$J_0^e(t,x)$ is defined by \ref{H8}. 
Then we represent
\begin{align}
&\int\limits_0^T\int \limits_{\Sigma} e^{-\lambda t}J_\ve^e (\phi^e_\ve-\hat u^e_\ve) \, dS(x)\, dt\notag\\
&=\int\limits_0^T\int \limits_G   \Delta \Psi\left( \Pi^\omega_\ve \phi_\ve^e-\Pi^\omega_\ve\hat u_\ve^e\right)
dx\, dt
+\int\limits_0^T\int \limits_G \nabla \Psi\cdot \nabla\left( \Pi^\omega_\ve \phi_\ve^e- \Pi^\omega_\ve\hat u_\ve^e\right) dx\, dt,
\end{align}
and pass to the limit using \eqref{Psi}, \eqref{gradPsi}, \eqref{conv_gradPue}, \eqref{conv_grad_phie}:
\begin{align}
\int\limits_0^T\int \limits_{\Sigma} e^{-\lambda t}J_\ve^e (\phi^e_\ve-\hat u^e_\ve) \, dS(x)\, dt \to
\int\limits_0^T\int \limits_{\Sigma} e^{-\lambda t}J_0^e (\phi^e_0-\hat u^e_0) \, dS(x)\, dt.
\end{align} 

Summarizing, from \eqref{semi_v}, \eqref{semi_g}, \eqref{85}, \eqref{87} we pass to the limit in \eqref{almost_ready_forlimit} and obtain
\begin{comment}
  &\frac{c_m}{2}  \int\limits_{\Omega}\int\limits_{G}\hat v^2_0(T) \, dx \, d\mu(\omega)
	+\frac{1}{2} \int\limits_{\Omega}\int\limits_{G}\hat g^2_0(T) \, dx\, d\mu(\omega)
	\nonumber\\
	-&\frac{c_m}{2}  \int\limits_{\Omega}\int\limits_{G} v^2_0(0) \, dx\, d\mu(\omega)
	-\frac{1}{2} \int\limits_{\Omega}\int\limits_{G} g^2_0(0) \, dx\, d\mu(\omega)\nonumber\\  
\end{comment}
\begin{align}
& c_m\int\limits_{\Omega}\int\limits_{G_T} \partial_t\hat v_0  
(\hat v_0 -[\phi_0])\, dx\, d\mu(\omega)\, dt
	+ \int\limits_{\Omega}\int\limits_{G_T} \partial_t\hat g_0  (\hat g_0-\psi)\, dx\, d\mu(\omega)\, dt\nonumber\\
&+ \int\limits_{\Omega}\int\limits_{G_T}
	 \hat I_{\rm ion}([\phi_0], \psi) 
	\left(\hat v_0- [\phi_0] \right)dx\, d\mu(\omega)\, dt	\nonumber \\   
	&+\int\limits_{\Omega}\int\limits_{G_T}
	\bigg(
	\theta [\phi_0] +ae^{-\lambda t}-(\lambda+b)\psi\bigg)
	(\psi-\hat g_0)   dx\, d\mu(\omega)\, dt \nonumber \\   
	&\leq	
 %\textcolor{red}{\int\limits_0^T\int\limits_{F_\ve \cup G_\ve}\hat\sigma_\ve(|\nabla \phi_\ve|)\nabla \phi_\ve \cdot (\nabla \phi_\ve-\nabla \hat u_\ve) dxdt }
 \int\limits_{G_T}\hat\sigma_{\rm hom}(\nabla \phi_0^e)\cdot(\nabla \phi_0^e-\nabla \hat u_0^e) dx\,  dt  
  \label{moneq: limit} 
  -\int\limits_0^T e^{-\lambda t}\int \limits_{\Sigma} J_0^e (\phi_0^e-\hat u_0^e) \, dS(x)\, dt\\ 
  &+ \frac{1}{2}\int\limits_\Omega\int\limits_{G_T} r(\omega)\hat\sigma^i(|\partial_{x_1} \phi^i_0|)\partial_{x_1}\phi^{i}_0  \,\partial_{x_1} (\phi^{i}_0-\hat u_0^{i}) \, dx\, d\mu(\omega) \, dt.	 
\notag
\end{align}
By density, \eqref{moneq: limit} holds for test functions
\begin{align*}
&\phi^{i}_0\in L^4(G_T; L^4(\Omega,\mu)\cap {\cal K(A)}):\,\,\partial_{x_1} \phi^{i}_0 \in L^2(G_T;{\cal K_\mu (A)}), \,\, \phi^{i}_0=0 \,\, \mbox{on} \,\, S_0\cup S_L,\\
&\phi^e_0\in L^4(G_T):\,\,\nabla \phi^e_0 
	\in \bigl(L^2(G_T)\bigr)^3, \,\, \phi^e_0=0 \,\, \mbox{on} \,\, S_0\cup S_L,\\
&\psi\in L^2(G_T;  L^2(\Omega,\mu)).
	%	\label{spaces for test}
\end{align*}
Now let us take in \eqref{moneq: limit} %as $\phi_0^e, \phi_0^i, \psi_0$ the 
test functions in the form
\begin{align*}
&\phi_0^e=\hat u_0^e + \delta \tilde \phi_0^e, \quad \phi_0^{i}=\hat u_0^{i} + \delta \tilde \phi_0^{i},\\
&\psi=\hat g_0 + \delta \tilde \psi, \quad 
[\phi_0]=\hat v_0 + \delta [\tilde \phi_0], 
\end{align*}
where $\delta\not= 0$ is a
%n arbitrary 
small parameter, and $\tilde \phi_0^i, \tilde \phi_0^e, \tilde \psi$ are 
%smooth functions such that $\tilde \phi_0^e=\tilde \phi_0^{i}=0$ on $S_0\cup S_L$; \  $[\tilde\phi_0]=\tilde\phi^{i}_0-\tilde\phi^{e}_0$ on $G$. 
of the same classes as the functions $ \phi_0^i, \phi_0^e, \psi$. We obtain 
\begin{align*}
& -	\delta c_m\int\limits_{\Omega}\int\limits_{G_T} \partial_t\hat v_0  
[\tilde\phi_0]\, dx\, d\mu(\omega)\, dt
	- 	\delta\int\limits_{\Omega}\int\limits_{G_T} \partial_t\hat g_0 	\tilde\psi\, dx\, d\mu(\omega)\, dt\nonumber\\
& - 	\delta\int\limits_{\Omega}\int\limits_{G_T}
\biggl((c_m\lambda-1)(\hat v_0+\delta[\tilde \phi_0]) +\frac{e^{2\lambda t}}{3} \, \left(\hat v_0+\delta[\tilde \phi_0]\right)^3 -\hat g_0-\delta\tilde \psi\biggr) 
	[\tilde\phi_0]dx\, d\mu(\omega)\, dt	\nonumber \\
&+\delta\int\limits_{\Omega}\int\limits_{G_T}
	\bigg(
	\theta \left(\hat v_0+\delta[\tilde \phi_0]\right) +ae^{-\lambda t}-(\lambda+b)(\hat g_0 + \delta \tilde \psi))\bigg)
	\tilde\psi   dx\, d\mu(\omega)\, dt \nonumber \\ 
	&\leq	
 %\textcolor{red}{\int\limits_0^T\int\limits_{F_\ve \cup G_\ve}\hat\sigma_\ve(|\nabla \phi_\ve|)\nabla \phi_\ve \cdot (\nabla \phi_\ve-\nabla \hat u_\ve) dxdt }
 \delta \int\limits_{G_T}\hat\sigma_{\rm hom}(\nabla (\hat u_0^e + \delta \tilde \phi_0^e))\cdot\nabla \tilde\phi_0^e dx\,  dt\\ 
  &+ \frac{\delta}{2}\int\limits_\Omega\int\limits_{G_T} r(\omega)\hat\sigma^i(|\partial_{x_1} (\hat u_0^i + \delta \tilde \phi_0^i)|)\partial_{x_1}(\hat u_0^i + \delta \tilde \phi_0^i)|)  \,\partial_{x_1} \tilde\phi^{i}_0 \, dx\, d\mu(\omega) \, dt.	 
		\nonumber \\ 
	&-\delta\int\limits_0^T e^{-\lambda t}\int \limits_{\Sigma} J_0^e \tilde\phi_0^e \, dS(x)\, dt.
\notag
\end{align*}
Then, dividing the last inequality
by $\delta$ and passing to the limits, as $\delta\to \pm 0$, we get
\begin{align}
\label{weak homogenized}
&  c_m\int\limits_{\Omega}\int\limits_{G_T} \partial_t\hat v_0  
[\tilde\phi_0]\, dx\, d\mu(\omega)\, dt
	+\int\limits_{\Omega}\int\limits_{G_T} \partial_t\hat g_0 	\tilde\psi\, dx\, d\mu(\omega)\, dt\nonumber\\
& +\int\limits_{\Omega}\int\limits_{G_T}
\biggl((c_m\lambda-1)\hat v_0 +\frac{e^{2\lambda t}}{3} \, \hat v_0^3 -\hat g_0\biggr) 
	[\tilde\phi_0]\, dx\, d\mu(\omega)\, dt	\nonumber \\   
&-\int\limits_{\Omega}\int\limits_{G_T}
	\bigg(
	\theta \hat v_0 +ae^{-\lambda t}-(\lambda+b)\hat g_0\bigg)
	\tilde\psi   \, dx\, d\mu(\omega)\, dt \nonumber \\   
&+	
  \int\limits_{G_T}\hat\sigma_{\rm hom}(\nabla \hat u_0^e)\cdot\nabla \tilde\phi_0^e dx\,  dt
  \nonumber \\ 
  &+ \frac{1}{2}\int\limits_\Omega\int\limits_{G_T} r(\omega)\hat\sigma^i(|\partial_{x_1} \hat u_0^i|)\partial_{x_1}\hat u_0^i   \,\partial_{x_1} \tilde\phi^{i}_0 \, dx\, d\mu(\omega) \, dt	 
		\nonumber \\ 
	&-\int\limits_0^T e^{-\lambda t}\int \limits_{\Sigma} J_0^e \tilde\phi_0^e \, dS(x)\, dt =0
\end{align}
By taking  $\tilde \psi=0$ and $\tilde \phi_0^e=\tilde \phi_0^{i}=\tilde\varphi_0$ in \eqref{weak homogenized} (with an arbitrary $\tilde\varphi_0\in C^2(G_T)$ that equals zero on $S_0\cup S_L$),  we derive that for almost all $t\in [0,T]$,
\begin{comment}
    \begin{align}
  &\int\limits_{G_T}\hat\sigma_{\rm hom}(\nabla \hat u_0^e)\cdot\nabla \varphi_0 dx\,  dt + \frac{1}{2}\int\limits_\Omega\int\limits_{G_T} r(\omega)\hat\sigma^i(|\partial_{x_1} \hat u_0^i|)\partial_{x_1}\hat u_0^i   \,\partial_{x_1} \varphi_0 \, dx\, d\mu(\omega) \, dt	 
		\nonumber \\ 
	&-\int\limits_0^T e^{-\lambda t}\int \limits_{\Sigma} J_0^e \varphi_0^e \, dS(x)\, dt =0,
\end{align}
\end{comment}
the function $\hat u_0^e\in {\cal H}_e$ solves the following problem
\begin{align}
\label{hat_u^e}
&\frac{1}{2}\int_\Omega\int_{G} r(\omega)\hat \sigma^i\left(|\partial_{x_1} (\hat v_0+\hat u_0^e)|\right) \partial_{x_1} (\hat v_0+\hat u_0^e) \partial_{x_1} \varphi_0 \, dx\, d\mu(\omega) \notag \\
&
+\int_{G} \hat \sigma_{\rm hom}^e(\nabla \hat u_0^e)  \cdot \nabla \varphi_0 \, dx
- e^{-\lambda t}\int_{\Sigma} J_0^e \varphi_0 \, dS(x)=0, \quad \forall 
\varphi_0\in {\cal H}_e.
\end{align}
%For $t\in[0,T]$, 
We introduce the operator $\hat {\mathcal L}_{\rm hom}(t, \cdot): {\cal H}_i \to {\cal H}^\prime_i$ by setting
\begin{align}
\label{def:oper-Lhom-hat}
\hat{\mathcal L}_{\rm hom}(t, \hat v)=-\frac{r(\omega)}{2}\partial_{x_1} 
\left( \hat \sigma^i\bigl(|\partial_{x_1}(\hat v+\hat u^e)|\bigr) \partial_{x_1}(\hat v+\hat u^e)\right),
\end{align}
and rewrite \eqref{weak homogenized} in the following form
\begin{align}
\label{weak-hom-hatL}
&  c_m\int\limits_{\Omega}\int\limits_{G_T} \partial_t\hat v_0  
[\tilde\phi_0]\, dx\, d\mu(\omega)\, dt
	+\int\limits_{\Omega}\int\limits_{G_T} \partial_t\hat g_0 	\tilde\psi\, dx\, d\mu(\omega)\, dt\nonumber\\
& +\int\limits_{\Omega}\int\limits_{G_T}
\biggl((c_m\lambda-1)\hat v_0 +\frac{e^{2\lambda t}}{3} \, \hat v_0^3 -\hat g_0\biggr) 
	[\tilde\phi_0]dx\, d\mu(\omega)\, dt\\   
&-\int\limits_{\Omega}\int\limits_{G_T}
	\bigg(
	\theta \hat v_0 +ae^{-\lambda t}-(\lambda+b)\hat g_0\bigg)
	\tilde\psi   dx\, d\mu(\omega)\, dt \nonumber \\   
&+	\int\limits_0^T(\hat {\mathcal L}_{\rm hom}(t, \hat v), [\tilde\phi_0])_{L^2(G; L^2(\cal M, \mu))}\, dt =0.\nonumber
\end{align}
Taking in \eqref{weak-hom-hatL} successively $\tilde \psi=0$ and $\tilde \phi_0^e=\tilde \phi_0^{i}=0$,  since  $\tilde \phi_0^{i,e} (t,x,\ \cdot \  )\in {\cal K(A)}$, we obtain
the weak formulation of the following homogenized problem in terms of $\hat v_0, \hat g_0$:
\begin{align}
\label{homogenized_hat}
& c_m\partial_t \hat v_0 
+\hat {\mathcal L}_{\rm hom}(t, \hat v_0)+ {\bf P}_{\cal K(\A )} \hat I_{\rm ion}(\hat v_0, \hat g_0)=0
&\mbox{in}\,\,  &G_T\times \cal M,\nonumber\\
& \partial_t\hat g_0 + (\lambda+b)\hat g_0 -\theta \hat v_0
-ae^{-\lambda t}=0 & \mbox{in} \ &G_T \times {\cal M},
\\
&\hat v_{0}(0,x) =V_0(x), \,\, \hat g_{0}(0,x,\omega)=G_0(x,\omega)\,\,&\mbox{in}\,\,  &G\times \cal M,\nonumber
\end{align}
where $\hat I_{\rm ion}$ is defined by \eqref{def:I_ion_hat}.

Performing the change of unknowns 
$ \hat u_0^e=e^{-\lambda t}u_0^e$, $\hat u_0^{i}=e^{-\lambda t}u_0^{i}$, $\hat v_0=e^{-\lambda t}v_0$, $\hat g_0=e^{-\lambda t}g_0$ in \eqref{homogenized_hat}, we arrive at the homogenized problem \eqref{eq:hom-prob}. Since the proof for $t=T$ can be repeated for all $t\in[0,T]$, it completes the proof of Theorem \ref{th:main-short}. \hfill $\square$

%Now integrating by parts in the first four terms of \eqref{weak homogenized} and

%%%%%%%%%%%%%%%%%%%%%%%%%%%%%%%%%%%%%%%%%%%%%%%%%%%%
\subsection{Well-posedness of the limit problem}
\label{sec:existence-eff-prob}
To prove that the limit problem \eqref{eq:hom-prob} possesses a unique solution, we rewrite \eqref{homogenized_hat} in the following form
\begin{align}
	\label{eq:abstract-parabol-eff}
	%&\partial_t \begin{pmatrix} \hat v_0 \\[6mm] \hat g_0 \end{pmatrix} + B_1(t, \hat v_0, \hat g_0) + B_2(t, \hat v_0, \hat g_0)+ B_3(t, \hat v_0, \hat g_0)= F_0(t), \quad & \mbox{in} \ &G_T \times {\cal M}\\
&\partial_t W_0+ B_1(t, W_0) + B_2(t, W_0)+ B_3(t, W_0)= F_0(t) \quad & \mbox{in} \ &G_T \times {\cal M}, \nonumber\\ 
	&W_0(0,x, \omega)=W_0^0(x,\omega) \quad & \mbox{in} \ &G \times {\cal M}, 
\end{align}
where
\begin{align*}
	W_0 &=\begin{pmatrix} \hat v_0\\[2mm] \hat g_0 \end{pmatrix}, \quad 
 W_0^0=\begin{pmatrix} V_0\\[2mm] G_0 \end{pmatrix},\\ 
	B_1(t, W_0)&:=\begin{pmatrix}
		\displaystyle \frac{1}{c_m}\hat {\mathcal L}_{\rm hom}(t, \hat v_0) -
		\frac{1}{c_m}\hat {\mathcal L}_{\rm hom}(t, 0)\\[4mm]
		0
		\end{pmatrix},\\
	%\label{eq:B2eff}
	B_2(t, W_0) &:=
	\begin{pmatrix}
		\displaystyle
		\frac{e^{2\lambda t}}{3c_m} \hat v_0^3 \\[4mm] 0
	\end{pmatrix},	\quad B_3(t, W_0):=\begin{pmatrix}
	\displaystyle  \big(\lambda -\frac{1}{c_m}\big)\hat v_0 - \frac{1}{c_m} {\bf P}_{\cal K(\A )} \hat g_0\\[4mm]
	(\lambda+b) \hat g_0 -\theta \hat v_0 
\end{pmatrix},\\
	F_0(t)& := \begin{pmatrix}
		\displaystyle
		-\frac{1}{c_m}\hat {\mathcal L}_{\rm hom}(t, 0) \\[2mm]
		ae^{-\lambda t} 
	\end{pmatrix}.
\end{align*}
Recall that  
\begin{align}
{\cal H}_i=\{ u^i\in L^2(G; {\cal K_\mu(\A )})\ : \partial_{x_1} u^i\in L^2(G; {\cal K_\mu(\A )}), \ u^i|_{S_0\cup S_L}=0 \},  \notag
%	&{\cal H}'_i=\{v=\partial^2_{x_1 x_1} u^i, u^i\in {\cal H}_i\}.\notag
\end{align}
with the norm given by
\begin{align*}
	\|u\|_{{\cal H}_i}^2 = \int\limits_\Omega \int_G |u|^2 \, dx\, d\mu(\omega)
	+ \int\limits_\Omega \int_G |\partial_{x_1} u|^2 \, dx\, d\mu(\omega).
\end{align*}
We also introduce the following functional spaces:
\begin{align*}
	{\cal H}&= L^2(G; {\cal K_\mu(\A )}) \times  L^2(G; L^2(\Omega,\mu)), \nonumber\\
	{\cal V}_1& ={\cal H}_i \times  L^2(G; L^2(\Omega,\mu)).\nonumber\\
	{\cal V}_2& =L^4(G; L^4(\Omega,\mu)\cap{\cal K(\A )}) \times   L^2(G; L^2(\Omega,\mu)).\nonumber
	\end{align*}
One can check that the operators $B_i$ are monotone and satisfy the conditions of 
%The existence of a unique solution to problem \eqref{eq:abstract-parabol} follows from 
%Applying the result of 
Theorem 1.4 in \cite{lions1969quelques} and Remark 1.8 in Chapter 2 (see also Theorem 4.1 in \cite{showalter2013monotone}) which proves the existence of a unique solution to problem \eqref{eq:abstract-parabol-eff}.

\appendix

% \section{Ergodic theory}
% Equivalent defintions of ergodicity:

% Proposition 2.14 in \cite{einsiedler2011ergodic} or Theorem 1.3 in \cite{walters2007ergodic}.

% https://proofwiki.org/wiki/Equivalence_of_Defintions_of_Ergodic_Measure-Preserving_Transformation#Definition_1_implies_Definition_2

\section{Examples of random geometry}
\label{sec:examples}
\label{sec:examples_random_geometry}
{In Sections \ref{sec:example-without-shifts}, \ref{sec:example-with-shifts}, we present examples of a probability space $(\Omega, \sigma, \cal P)$, a dynamical system $T_{x'}$, and a random set $A(\omega)$ together with $\A $  providing randomly placed axons of a finite number of different radii. We refer to \cite{pettersson2023bidomain} for the case of identical randomly distributed axons. } 
\subsection{Randomly distributed disks of two different radii}
\label{sec:example-without-shifts}
In what follows, we will use the ergodicity of the random shifts.
\begin{lemma}[Random shifts of a periodic structure]
\label{ex:shift}
Given a sample space $\Omega =[0,1)^3$ with the Lebesgue measure and the Borel $\sigma$-algebra ${\mathcal F}$, define the shifts $T_x$ as follows:
		\begin{align}
        \label{eq:shifts}
			T_x \omega = \omega + x -[\omega +x], \quad \omega \in \Omega, \,\, x\in \mathbb R^3.
		\end{align}
Here $[x]$ is the floor of $x\in \mathbb R^3$ (the integer part of $x$). Then the group of transformations $T_x:\Omega \to \Omega$ is ergodic.
\end{lemma}
\begin{proof}
Note that for the integer shifts $x\in \mathbb Z^3$, the semigroup becomes the identity transformation $T_x \omega = \omega$.
\noindent
Clearly, $T_x$ satisfies the semigroup properties. Indeed, $T_0 \omega =0$, and
\begin{align*}
	&T_y T_x \omega = T_y(\omega + x -[\omega +x]) \\
	&= \omega + x -[\omega +x] + y - [\omega + x -[\omega +x] + y]\\
	&= \omega + x + y - [\omega + x + y] = T_{y+x}\omega.
\end{align*}
To check that the transformation is measure-preserving it is sufficient to do it for an arbitrary cell $[a_1,b_1)\times [a_2,b_2)\times[a_3,b_3) \subset [0,1)^3$ (generating algebra). This is clearly true because the Lebesgue measure is translation-invariant.
		
Let us check that $T_x$ is ergodic. Take any cell $A=[a_1,b_1)\times [a_2,b_2)\times[a_3,b_3) \subset [0,1)^3$ invariant with respect to the transformation $T_x$, for any $x\in \mathbb R^3$. If $x\in \mathbb Z^3$, any cell is an invariant set because, in this case, $T_x=\mathrm{Id}$. For non-integer $x$, the only invariant sets are the empty set and $[0,1)^3$. No set of the form $T_x A$ with at least one $x_i \notin \mathbb Z$ is invariant. Suppose that the first component $x_1$ is not aninteger. Then the projection of the image of $A$ under $T_{x}$ (which is a finite union of cells) on the $x_1$-axis does not coincide with $[a_1, b_1)$. Therefore, the only sets that are invariant for any $x\in \mathbb R^3$ are the empty set and $[0,1)^3$, and the group of transformations $T_x$ is ergodic. 
\end{proof}
	
		To construct a random structure with disks of two different radii, we start by defining a probability space $(\Omega,{\mathcal F},{{\mathcal P}})$.
		Let the sample space be $\Omega=\Omega'\times Y$, where $\Omega'={\lbrace{ 0,1,2 \rbrace}}^{\mathbb{Z}^2}$,	$Y=[0,1)^2$. {The elements of the sample space are $\omega=(\omega',\xi)$ describe a system of randomly independently distributed disks centered at the points  $\mathbb{Z}^2+(\frac{1}{2}, \frac{1}{2} )- \xi$.} 
		Here, $\omega'\in \Omega'$  determines the scattering of the disks 
		$$
		D_k(z)=\biggl\{x'\in \mathbb{R}^2: \left| x'-z- (1/2, 1/2) \right|<r_k \biggr\},
		\quad k=1,2,
		$$
		of two distinct radii $r_1<r_2< 1/2$ centered at the {nodes of the lattice} $\mathbb{Z}^2+(1/2, 1/2)$, and is given by
		\begin{align*}
			\omega'(z) = \begin{cases}
				\displaystyle
				0 \quad \mbox{if there is no disk centered at}\ z+(1/2, 1/2),\  \\ 
				1 \quad \mbox{if there is a disk}\quad D_1(z),\qquad\qquad\qquad\qquad\\ 
				2 \quad \mbox{if there is a disk}\quad D_2(z).\qquad\qquad\qquad\qquad
			\end{cases}
		\end{align*}
		{We assume that the probabilities of these independent events are positive, and they will be denoted by $p_0, p_1$, and $p_2$, respectively}, such that $p_0+p_1+p_2=1$. One can interpret $\omega'$ as coloring the cells of the grid $\mathbb Z^2$ in three different colors.
		The random shift of all the disks is given by $\xi\in Y=[0,1)^2$.
		
		In order to construct a $\sigma$-algebra  $\sigma^\prime\subset 2^{\Omega'}$ and a probability measure $d{\cal P^\prime}(\omega')$, for an arbitrary finite subset $C_N\subset \mathbb{Z}^2$, $C_N={\lbrace{ z_k;\, k=1,\ldots, N \rbrace}}$, we consider a cylindrical set $\tilde\omega_N^\prime=\{\omega^\prime: \,\, \omega^\prime(z_k)=\omega^\prime_k,\, k=1,\ldots, N\}$, with
		$\omega^\prime_k\in\{0,1,2\}$, and define its probability by 
		$$
		{\cal P^\prime}(\tilde\omega_N')=p_0^{N_0}p_1^{N_1}p_2^{N_2}, \quad \mbox{where}\,\, N_i= \#\{k:\, \omega^\prime_k=i\}.
		$$
		By the Kolmogorov extension theorem \cite{shiryaev2016probability}, this probability measure can be extended to $d{\cal P^\prime}$ with the $\sigma$-algebra $\sigma^\prime$ generated by cylindrical sets.

		We define then a $\sigma$-algebra on $\Omega$ by setting ${\mathcal F}=\sigma'\times \mathcal B$, with $\mathcal B$ being the Borel $\sigma$-algebra on the unit cell $Y=[0,1)^2$, and defining the probability measure on $\sigma$ as $d{\cal P}=d{\cal P^\prime} d\xi $, where $d\xi$ is the Lebesgue measure on $Y$.
		
		In the next step, we introduce a dynamical system  $T_{x'}:\Omega \rightarrow \Omega$ by setting
		\begin{align}
			\label{dynsys}
			T_{x'}\omega=\big(\omega'(z+[x'+\xi]), \, x'+\xi-[x'+\xi]\big),
		\end{align}
		where $[x' + \xi]$ is the integer part of $x' + \xi$. The mapping \eqref{dynsys} is a group of measure preserving transformations such that (i)--(iii) are satisfied. 
		Indeed, since $[x'+\xi]=x+[\xi]$ for $x\in \mathbb Z^2$, $T_{x'}$ satisfies the group property:
		\begin{align*}
			T_0\omega &= (\omega'(z), \xi),\\
			T_{y'}(T_{x'} \omega)&= T_{y'}\big(\omega'(\tilde z), \, \tilde \xi\big), \,\, \tilde z = z+[x'+\xi], \, \tilde \xi = x'+\xi-[x'+\xi],\\
			T_{y'}(T_{x'} \omega)&= \big(\omega'(\tilde z + [y'+\tilde \xi]), \, y'+\tilde \xi - [y' + \tilde \xi]\big)\\
			&= \big(\omega'(z + [y'+x'+\xi]), \, y'+x'+\xi - [y' + x'+\xi]\big)= T_{y'+x'}\omega.
		\end{align*}
		The measurability of $T_x\omega$ follows from the definition, as a composition of measurable maps.
		\begin{lemma}
			$T_{x'}:(\Omega,{\mathcal F},{{\mathcal P}}) \to (\Omega,{\mathcal F},{{\mathcal P}})$ given by \eqref{dynsys} is ergodic.
		\end{lemma}
		\begin{proof}
			Take any set $A\subset {\mathcal F}$ invariant with respect to $T_x$, that is $T_x A = A$ for any $x\in \mathbb R^3$. From \eqref{dynsys} we see that
			\begin{align*}
				T_{x'}\omega &= \big(S_{[x'+\xi]}'\omega', S_{x'}''\xi\big), \\
				S_{[x'+\xi]}'\omega' &= \omega'(\cdot + [x'+\xi]), \quad
				S_{x'}''\xi = x'+\xi -[x'+\xi].
			\end{align*}
			One can see that both shifts $S_{x'}':(\Omega', {\mathcal F}, {\mathcal P})\to (\Omega',{\mathcal F}, {\mathcal P})$ and $S_{x'}'': ([0,1)^2, {\mathcal B}, \lambda) \to ([0,1)^2, {\mathcal B}, \lambda)$ are measure preserving groups of transformations, $x'\in \mathbb R^2$. If $A \subset {\mathcal F}$ is $T_{x'}$-invariant, then 
			\begin{align*}
				A &=\{\omega=(\omega', \xi): \, \omega \in A', \, \xi \in A''\},\\
				T_{x'}A &= A \quad \Rightarrow \quad S_{x'}'A'= A', \,\, S_{x'}''A''= A''.
			\end{align*}
			Since the measure $d{\mathcal P}=d{\mathcal P}'\times d\lambda$ is a product measure, a set $A$ of pairs $(\omega', \xi)$ is invariant if $A'$ and $A''$ are invariant with respect to the corresponding shifts. 
			We have already proved that the shifts of the periodic structure are ergodic (see Example \ref{ex:shift}), so $S_{x'}''$ is ergodic. As for $S_{x'}'$, we need to prove the ergodicity only for integer $x'$.  
			Let us show that $S_{x'}'$ is ergodic for $x'\in \mathbb Z^2$.
			Take an invariant w.r.t $S_{x'}'$ set $A'$, and suppose that ${\mathcal P}'(A')$ is not zero or one. There exists a cylinder set $A_\delta'$ approximating $A'$, that is for any $\delta<0$, 
			\begin{align}
				\label{eq:inv-0}
				{\mathcal P}(A'\Delta A_\delta') \le \delta, \quad {\mathcal P}(S_{x'}' A'\Delta S_{x'}' A_\delta') \le \delta,
			\end{align} 
			where we have used the measure preserving property of  $S_{x'}'$.  
			Now, since $A'$ is $S_{x'}'$-invariant, 
			\begin{align}
				\label{eq:inv-1}
				{\mathcal P}'(A' \cap S_{x'}'A') = {\mathcal P}'(A').
			\end{align}
			At the same time, there exists (maybe large) $x_0 \in \mathbb Z^2$ such that the cylinder sets $A_\delta'$ and $S_{x_0}' A_\delta'$ are supported by the finite sets of points in $\mathbb Z^2$ that do not intersect. The events $A_\delta'$ and $S_{x_0}' A_\delta'$ have the same measure since the number of disks with different radii is not affected by a shift. That implies
			\begin{align}
				\label{eq:inv-2}
				{\mathcal P}'(A_\delta'\cap S_{x'}' A_\delta')= {\mathcal P}'(A_\delta')^2.
			\end{align}
			Combining \eqref{eq:inv-0}--\eqref{eq:inv-2} we get 
			\begin{align*}
				{\mathcal P}'(A')={\mathcal P}'(A_\delta')^2 + O(\delta), \, \delta \to 0 \quad \Rightarrow \quad {\mathcal P}'(A')=0 \ \mbox{or} \ 1.
			\end{align*}
			Since ${\mathcal P}$ is a product measure, the ergodicity of $S_{x'}'$ and $S_{x'}''$ implies the ergodicity of $T_x$ with respect to ${\mathcal P}$. 
		\end{proof}
		Next proposition specifies the prototypes $\A _k$, $k=1, 2$ of the random sets $A_k(\omega)$.
\begin{lemma}
\label{lm:random-sets}
Let $A(\omega)$ be a collection of disks of two different radii $r_1, r_2$:
$$
A(\omega)=\bigcup_{k=1}^2\bigcup_{j\in \mathbb{Z}^2: \,\omega'(j)=k} (D_k(j)-\xi), \quad 
D_k(j) = \{x': |x' - j - (1/2, 1/2)| \le r_k\}.
$$
Then
\begin{align}
\label{eq:cal(D)}
	\A _k=\lbrace{ \omega\in\Omega: \omega'(0)=k, \, \xi \in D_k(0) \rbrace},\, k=1,2.
\end{align}

		\end{lemma}
		\begin{proof}
A given point $x'\in \mathbb R^2$ either belongs or not to a disk $D_k(j)-\xi$ for some $k$ and $j$. Suppose first that $x'\in (D_k(j)-\xi)$. Then $(x'+ \xi) \in D_k(j)$, and $[x'+\xi]=j$.
Then
\begin{align*}
		T_{x'} \omega = (\omega'(z+j), x' + \xi - j) \in \A _k 
\end{align*}
since $\omega'(j)=k$ and $z\in \mathbb Z^2$.
			
If $x$ does not belong to $D_k(j)-\xi$ for any $j, k$, a similar argument leads to $T_{x'}\omega \notin \A _k$, that proves Lemma \ref{lm:random-sets}.
		\end{proof}
The construction in Lemma \ref{lm:random-sets} gives a random placement of disks in $\mathbb R^2$. Figure \ref{fig:random1} illustrates this example for specific $p_0, p_1$, and $p_2$.
		
Note that due to \eqref{eq:cal(D)}, the limit densities \eqref{dens} can be computed explicitly
\begin{align}
	\label{eq:densities}
	{\cal P}(\A _k)= \pi r_k^2 p_k,
\end{align} 
since the probability to have a disk of radius $r_k$ at $z=0$ is $p_k$, and the Lebesgue measure of $D_0^k$ is $\pi r_k^2$.

To write down the limit problem \eqref{eq:hom-prob} in the example under consideration, we start by characterizing the Palm measure and computing the projection on $\mathcal K(\A)$. Recall Lemma \ref{lemma_radii} and take $f=1$ on $\bar{\A_k}$ and zero otherwise. Equality \eqref{eq:formulka} takes the form
\begin{align*}
\mathcal P (\A_k) = \frac{r_k}{2} \int \limits_{\bar{\A_k}} d\mu(\omega),
\end{align*}
that yields $\mu(\A_k) = 2\pi r_k p_k$.
We also introduce components of the prototype set $\mathcal M$
\begin{align}
\label{eq:prototype-M}
\mathcal M = \mathcal M_1 \cup \mathcal M_2, \,\, \mathcal M_k = \lbrace{ \omega\in\Omega: \omega'(0)=k, \,\, \xi \in \partial D_k(0)\rbrace}, \, k=1, 2.
\end{align}
Next, we compute the projection of the limit function $g_0(x,\omega)$ on $\mathcal K(\A)$ in $L^2(\Omega, \mu)$. Choosing an orthonormal basis in $\mathcal K(\A)$,
$e_k = \chi_{\mathcal M_k}(\omega)/\sqrt{\mu(\mathcal M_k)}, \, k=1,2$, the projection takes the form
\begin{align*}
\mathbf P_{\mathcal K(\A)} g_0(t, x, \omega)
= \sum_{k=1}^2 \frac{\chi_{\mathcal M_k}}{\mu(\mathcal M_k)} \int \limits_{\mathcal M_k} g_0(t, x, \omega) \, d\mu(\omega).
\end{align*}
The Palm measure $\mu(\omega)$ is
\begin{align*}
d\mu(\omega)
= \sum_{k=1}^2\delta(\omega'(0) - k) \delta(|\xi|-r_k)\, d \mathcal P'(\omega')ds_\xi,
\end{align*}
where $ds_\xi$ denote the arc length measure on $\partial D_{r_k}(0)$.
By the definition of $\mathcal M_k$ \eqref{eq:prototype-M}, 
\begin{align}
\mathbf P_{\mathcal K(\A)} g_0(t, x, \omega)
= \sum_{k=1}^2 \frac{\delta(\omega'(0) - k) \delta(|\xi|-r_k)}{\mu(\mathcal M_k)}
\int \limits_{\omega'(0)=k} \int \limits_{\partial D_{r_k}(0)} g_0(t, x, \omega', \xi)\, d\mathcal P '(\omega')ds_\xi.
\end{align}
Thus,
\begin{align}
\mathbf P_{\mathcal K(\A)} g_0(t, x, \omega)\big|_{\omega \in \mathcal M_k} = g_0^k(t, x) = 
\frac{1}{2\pi r_k p_k}
\int \limits_{\omega'(0)=k} \int \limits_{\partial D_{r_k}(0)} g_0(t, x, \omega', \xi)\, d\mathcal P '(\omega') ds_\xi.
\end{align}
Denoting $u_0^i(t, x, \omega)\big|_{\omega\in \mathcal M_k}=u_0^{ik}(t,x)$, we obtain the following limit problem
\begin{align}
	\label{eq:hom-prob-k}
	& c_m\partial_t v_{0}^k + I_{ion}(v_{0}^k, g_{0}^k)  = \frac{r_k}{2} \partial_{x_1}(\sigma^i(|\partial_{x_1} u^{ik}_0|))  & &\mbox{in}\, G_T,\nonumber\\
	& \sum_{k=1}^2 \mathcal P (\A_k) (c_m\partial_t v_{0}^k + I_{ion}(v_{0}^k, g_{0}^k))  = -\di{\sigma_{hom}^e(\nabla u_0^e)}  & &\mbox{in}\, G_T,\nonumber\\
	%\label{eq:hom-prob}
	&\partial_t g_{0}  = \theta v_{0}+a - b\, g_{0},  & &\mbox{in}\, G_T \times \mathcal M\\
	&v_{0}^k=u_{0}^{ik} - u_0^e & &\mbox{in}\, G,\ \nonumber\\ 
	&u_0^e = u_{0}^{ik}=0& &\mbox{on}\, S_0\cup S_L,\nonumber\\
	&\sigma_{hom}^e(\nabla u_0^e) \cdot \nu = J_0^e &&\mbox{on}\,  \Sigma,\nonumber\\
	&v_{0}^k(0,x) =V_0(x), \,\, g_{0}(0,x,\omega)=G_0(x, \omega)  
	\hskip -0.5cm &  &\mbox{in}\, G\times \mathcal M.\nonumber
\end{align}

\begin{figure}[htbp]
%\begin{subfigure}[t]{0.45\textwidth}
\centering
\includegraphics[clip, trim=7cm 10.5cm 6cm 10.5cm, width=0.35\textwidth]{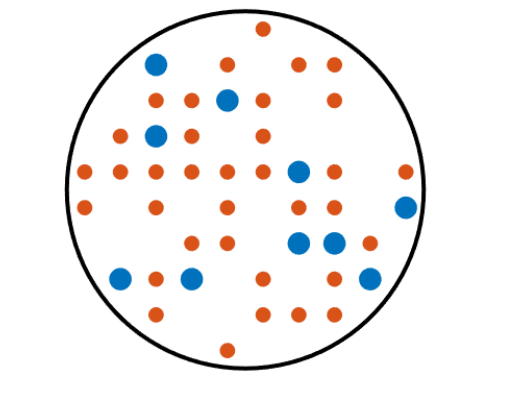}
%\caption{}
%\label{fig:random}
%\end{subfigure}
% \begin{subfigure}[t]{0.45\textwidth}
% \centering
% \includegraphics[clip, trim=7cm 10.5cm 6cm 10.5cm, width=0.8\textwidth]{poisson.pdf}
% \caption{}
% \label{fig:poisson}
% \end{subfigure}
\caption{Random distribution of discs in Section \ref{sec:example-without-shifts}: $p_0=1/3$ (no disk), $p_1=1/2$ (orange disk with radius $r_1=0.2$), $p_2=1/6$ (blue disk with radius $r_2 = 0.3$).} 
%(b) Poisson point process in Section \ref{sec:point-process}, discs with radius $r=0.2$, $\lambda=10$.}
\label{fig:random1}
\end{figure}
		
%\end{example}
\subsection{Disks with a finite number of different radii and random shifts}
\label{sec:example-with-shifts}
With small modifications, we generalize the example in the previous section and include random independent shifts inside the grid. 
	% \begin{example}
	% 	[Randomly distributed disks of a finite number of different radii and random shifts]
	% 	$ $
	Let $R<1/2$, and let {$\eta=(\eta_1, \eta_2, \ldots)$} be uniformly distributed random variables with values in $B_{1/2-R}(0)$ (ball of radius $1/2-R$ centered at the origin).
		
		We introduce a sample space $\Omega=\Omega'\times \Omega''\times Y$, where $\Omega'$ will determine the scattering of the disks of $N$ different radii (equivalently, the coloring of the cells in the periodic grid in $N+1$ colors); $\Omega''$ represents the random shifts of disks inside the cells, and	$Y=[0,1)^2$ is an auxiliary shift of the whole structure to ensure that the group property is satisfied. 
		
		{The elements $\omega'$ of the sample space $\Omega'$ are functions from $\mathbb Z^2$ to $\{0, 1, \ldots, N\}$, and describe a system of randomly independently distributed disks centered at the points  $\mathbb{Z}^2+(\frac{1}{2}, \frac{1}{2} )- \xi$. } 
		As before, we denote the disks by
		$$
		D_k(z)=\biggl\{x'\in \mathbb{R}^2: \left| x'-z- (1/2, 1/2) \right|<r_k \biggr\},
		\quad k=1,\ldots, N. 
		$$
		The coloring is defined by 
		\begin{align*}
			\omega'(z) = \begin{cases}
				\displaystyle
				0 \quad \mbox{if there is no disk centered at}\ z+(1/2, 1/2),\  \\ 
				k \quad \mbox{if there is a disk}\quad D_k(z), \,\, k=1, \ldots, N.\qquad\qquad\qquad\qquad
			\end{cases}
		\end{align*}
		We assume that the probabilities of these independent events are positive, and they will be denoted by $p_k$, $k=1, \ldots, N$, such that $\sum_{k=1}^N p_k=1$. 
		Next, the random shifts of each disk are represented by the elements of $\Omega'': \mathbb Z^2 \to B_{1/2-R}(0)$. 
		The random shift of all the disks is given by $\xi\in Y=[0,1)^2$.
		
		The sample space is then the product $\Omega = \Omega'\times \Omega'' \times Y$. 
		The sigma-algebra and a probability measure for $\Omega'$ are constructed in the previous example. 
		
		To use the Kolmogorov theorem and construct a sigma-algebra and a probability measure for $\Omega''$, we introduce the cylinder sets. For an arbitrary finite number of points on the grid $\{j_1, j_2, \ldots, j_N\}$, we consider a cylindrical set 
		\begin{align*}
			C_N (I_{j_1}, \ldots, I_{j_N}) = \{\eta_{j_1}(\omega) \in I_{j_1}, \ldots, \eta_{j_N}(\omega) \in I_{j_N}\}, \quad I_{j_k} \subset B_{1/2-R}(0),
		\end{align*}
		and define its probability by 
		$$
		{\cal P''}\{\eta_{j_1}(\omega) \in I_{j_1}, \ldots, \eta_{j_1}(\omega) \in I_{j_N}\} = \Pi_{k=1}^N \frac{|I_{j_k}|}{|B_{1/2-R}(0)|}.
		$$
		By the Kolmogorov extension theorem, ${\mathcal P}''$ can be extended to $d{{\mathcal P}''}$ with the $\sigma$-algebra $\sigma''$ generated by cylindrical sets.

		We define then a $\sigma$-algebra on $\Omega$ by setting ${\mathcal F}=\sigma'\times \sigma''\times \mathcal B$, with $\mathcal B$ being the Borel $\sigma$-algebra on the unit cell $Y=[0,1)^2$, and the probability measure on $\sigma$ as $d{\cal P}=d{\cal P^\prime} \, d{\mathcal P}'' \,d\xi $, where $d\xi$ is the Lebesgue measure on $Y$.
		
		Finally, we introduce a dynamical system  $T_{x'}:\Omega \rightarrow \Omega$:
		\begin{align}
			\label{dynsys-2}
			T_{x'}\omega=\big(\omega'(z+[x'+\xi]), \,
			\eta(z+[x'+\xi]), \, x'+\xi-[x'+\xi]\big),
		\end{align}
		where $[x' + \xi]$ is the integer part of $x' + \xi$. The mapping \eqref{dynsys-2} is a group of measure preserving transformations such that (i)--(iii) are satisfied, and $T_{x'}$ is ergodic. 
		For
		\begin{align}
			\label{eq:cal(D)-2}
			\A = \bigcup_{k=1}^N \A_k = \bigcup_{k=1}^N\lbrace{ \omega\in\Omega: \omega'(0)=k, \,\,
				\eta(0)\in B_{1/2 - R}(0), \,\, \xi \in D_k(0) + \eta(0)\rbrace},
		\end{align}
		we obtain a collection of disks of $N$ different radii $r_k$, $k=1, \ldots, N$:
		$$
	A(\omega)=\bigcup_{k=1}^N\bigcup_{j\in \mathbb{Z}^2: \,\omega'(j)=k} (D_k(j)+\eta(j)-\xi), \, 
		D_k(j) = \{\zeta: |\zeta - j - (1/2, 1/2)| \le r_k\}.
		$$
		Using \eqref{eq:cal(D)-2}, one can see that the limit densities are the same as in the previous example:
		$
		{\cal P}({\cal A}_k)= \pi r_k^2 p_k$.

Among other examples of random geometry are those involving Voronoi tesselation structures (see, e.g., Example 4 in \cite{bourgeat2003double}, and Example 3 in \cite{mikelic2023homogenization}). 
%%%%%%%%%%%%%%%%%%%%%%%%%%%%%%%%%%%%%%

\subsection{Monotonicity of the nonlinear terms}
\begin{lemma}
\label{lm:monotonicity}
Let $\sigma(\eta)$ satisfy \ref{H5}, that is
\begin{align*}
0<\underline \sigma \le \sigma(\eta)\le \overline{\sigma}, \quad \eta \frac{d}{d\eta}\sigma(\eta) + \eta >0.
\end{align*}
Then for all $\xi, \eta \in \mathbf R^3$ it holds
\begin{align}
\label{eq:equiv-monotonicity}
(\sigma(|\xi|)\xi - \sigma(|\eta|)\eta)\cdot (\xi-\eta)\ge \underline{\sigma} |\xi-\eta|^2,
\end{align}
with $\underline{\sigma}$ in assumption \ref{H5}.
\end{lemma}
\begin{proof}
Introduce $\zeta(t)=\eta + t(\xi-\eta)$ for $0\le t\le 1$ and consider $F(t)= \sigma(|\zeta(t)|)\zeta(t)$. Since $F(0)=\eta, F(1)=\xi$, \eqref{eq:equiv-monotonicity} is equivalent to the inequality
\begin{align*}
(F(1)-F(0))\cdot (\zeta(1)-\zeta(0))=F'(\zeta(t_0))\cdot(\xi-\eta)>0.
\end{align*}
Differentiating $F$ in $t$, we obtain
\begin{align*}
F'(\zeta(t))\cdot \zeta'(t)
&= \sigma'(|\zeta(t)|) \frac{\zeta'(t) \cdot \zeta(t)}{|\zeta(t)|} \zeta(t)\cdot \zeta'(t)
+ \sigma(|\zeta(t)|) |\zeta'(t)|^2\\
&= |\zeta(t)| \sigma'(|\zeta(t)|) \big(\mbox{proj}_{\zeta/|\zeta|} \zeta'(t)\big)^2 + \sigma(|\zeta(t)|)|\zeta'(t)|^2\\
&=\big(\mbox{proj}_{\zeta/|\zeta|} \zeta'(t)\big)^2 \big( |\zeta(t)| \sigma'(|\zeta(t)|) + \sigma(|\zeta(t)|)\big)\\
&+ \sigma(|\zeta(t)|) \big( |\zeta'(t)|^2- \big(\mbox{proj}_{\zeta/|\zeta|} \zeta'(t)\big)^2\big)
\end{align*}
where $\mbox{proj}_{\zeta/|\zeta|} \zeta'(t) = \frac{\zeta'(t) \cdot \zeta(t)}{|\zeta(t)|}$ is the projection of $\zeta'(t)$ on the unit vector $\zeta(t)/|\zeta(t)|$.
By assumption \ref{H5}, we can estimate the last expression from below, using $\zeta'(t)=\xi-\eta$:
\begin{align*}
F'(\zeta(t))\cdot \zeta'(t)\ge \underline{\sigma} \, |\zeta'(t)|^2 = \underline{\sigma} |\xi-\eta|^2.
\end{align*}
\end{proof}

% \section*{Acknowledgments}
% We would like to acknowledge the assistance of volunteers in putting
% together this example manuscript and supplement.

%%%%%%%%%%%%%%%%%%%%%%%%%%%%%%%%%%%%%%%%%%%
%%%%%%%%
%\appendixnotitle
{
\section{Nondimensionalization of the problem} \label{sec:dimension-analysis}
The asymptotic analysis of the microscopic problem is done under the assumption that a typical diameter of individual axons in the fascicle is much smaller than its length. In order to motivate the choice of the scaling parameter $\ve$ in the nonlinear dynamics on the membrane, one needs to perform the nondimensionalization of the problem. In the current section, we provide an example of how such an analysis can be done for specific applications of electrical stimulation. For the dimension analysis of the ionic current, we refer to \cite{rioux2013predictive}, and for the dimension analysis in the homogenization context to \cite{neu1993homogenization} and \cite{collin2018mathematical}.

Let $L_0$ be a characteristic length of a nerve, $T_0$ a characteristic time of signal propagation, $\Sigma_0$ a characteristic conductivity, and $U_0=V_{max}-V_{min}$ with $V_{min}$ and $V_{max}$ being characteristic upper and lower bounds for the membrane potential. We also introduce $G_0$, a characteristic value of the variable $g$. We set
\begin{align*}
    &\tilde{x}=\frac{x}{L_0}, \quad \tilde{t}=\frac{t}{T_0},
    \quad u^\alpha(x,t)= U_0\, \widetilde{u}^\alpha(\tilde{t}, \tilde{x}),\quad
    \sigma^\alpha(|\nabla u^\alpha|) = \Sigma_0\widetilde{\sigma}^\alpha(|\nabla_{\tilde x} \widetilde{u}^\alpha|), \, \alpha=i, e.
\end{align*}
We set also $I_{ion}(v, g) = I_0 \widetilde{I}_{ion}(\widetilde{v}, \widetilde{g})$ and $g= G_0 \tilde{g}$.
The cell membrane is modeled as a combined resistor and a capacitor, which results in the total current being decomposed into two parts 
\begin{align*}
c_m\partial_t v + I_{\rm ion}(v, g)
= - \sigma^i(|\nabla u^i|) \nabla u^{i}\cdot \nu.
\end{align*}
Nondimensionalizing the last equations yields
\begin{align*}
% \tilde{c}_m \frac{C_0U_0}{T_0}\partial_{\tilde{t}} \tilde{v} + I_0 \tilde{I}_{\rm ion}(\tilde{v}, \tilde{g}))= - \frac{\Sigma_0 U_0}{L_0}\tilde{\sigma}^i(|\nabla \tilde{u}|) \nabla \tilde{u}^{i} \cdot \nu\\
\frac{L_0}{\Sigma_0 U_0}\big(\frac{c_m U_0}{T_0}\partial_{\tilde{t}} \tilde{v} + I_0 \tilde{I}_{\rm ion}(\tilde{v}, \tilde{g})\big)= - \tilde{\sigma}^i(|\nabla_{\tilde{x}} \tilde{u}^i|) \nabla_{\tilde{x}} \tilde{u}^{i} \cdot \nu.
\end{align*}
We adopt the parameter values used in \cite{mandonnet2011role} for a numerical study of an axonal bipolar electrical stimulation (see also \cite{collin2018mathematical} for the typical parameter values used in cardiac electrophysiology). A typical diameter of an axon is taken to be $2\,\mu$m and a typical sample size $L_0=1$ cm. We introduce a small parameter $\ve$ as the ratio between a typical axon diameter and $L_0$, which gives $\ve$ of the order of $10^{-4}$. Taking $c_m = 0.6\, \mu$F/m, $U_0$ of order $10^{-1}$ V, $I_0$ of order $10^{-3}$ A/m$^2$, and $T_0=500\, \mu$s, we see that $c_m U_0/T_0$ and $I_0$ are of the same order of magnitude. The dimensionless coefficient $L_0 I_0/(\Sigma_0 L_0)$ is then of the same order as $\ve$, that is $10^{-4}$, which motivates the choice of the scaling parameter in \eqref{eq:orig-prob}. Note that since the FitzHugh-Nagumo model is phenomenological, the ODE for the variable $g$ is written in nondimensional form. 
}
%%%%%%%%%%%%%%%%%%%%%%%%%%%%%%%%%%%%%%%%%%%%%%%%
%%%%%%%%%%%%%%%%%%%%%%%%%%%%%

\bibliographystyle{siamplain}
\bibliography{refs-random-axons}

\begin{thebibliography}{10}

\bibitem{amar2013hierarchy}
{\sc M.~Amar, D.~Andreucci, P.~Bisegna, R.~Gianni, et~al.}, {\em A hierarchy of
  models for the electrical conduction in biological tissues via two-scale
  convergence: The nonlinear case}, DIFFERENTIAL AND INTEGRAL EQUATIONS, 26
  (2013), pp.~885--912.

\bibitem{BAR00}
{\sc P.~J. Basser and B.~J. Roth}, {\em New currents in electrical stimulation
  of excitable tissues}, Annu. Rev. Biomed. Eng., 2 (2000), pp.~377--397.

\bibitem{BenMroSaaTal2019}
{\sc M.~Bendahmane, F.~Mroue, M.~Saad, and R.~Talhouk}, {\em Unfolding
  homogenization method applied to physiological and phenomenological bidomain
  models in electrocardiology}, Nonlinear Analysis: Real World Applications, 50
  (2019), pp.~413--447.

\bibitem{bihoreau2023mathematical}
{\sc S.~Bihoreau, G.~Caluori, A.~Collin, P.~Ja{\"\i}s, M.~Legu{\`e}be, and
  C.~Poignard}, {\em Mathematical modeling of pulsed electric field cardiac
  ablation}, in RICAM-Workshop 4" Modeling and simulation of ablation
  treatments", 2023.

\bibitem{binczak2001ephaptic}
{\sc S.~Binczak, J.~Eilbeck, and A.~C. Scott}, {\em Ephaptic coupling of
  myelinated nerve fibers}, Physica D: Nonlinear Phenomena, 148 (2001),
  pp.~159--174.

\bibitem{blanc2007stochastic}
{\sc X.~Blanc, C.~Le~Bris, and P.-L. Lions}, {\em Stochastic homogenization and
  random lattices}, Journal de math{\'e}matiques pures et appliqu{\'e}es, 88
  (2007), pp.~34--63.

\bibitem{bokil2001ephaptic}
{\sc H.~Bokil, N.~Laaris, K.~Blinder, M.~Ennis, and A.~Keller}, {\em Ephaptic
  interactions in the mammalian olfactory system}, Journal of Neuroscience, 21
  (2001), pp.~RC173--RC173.

\bibitem{bourgeat2003double}
{\sc A.~Bourgeat, A.~Mikeli{\'c}, and A.~Piatnitski}, {\em On the double
  porosity model of a single phase flow in random media}, Asymptotic Analysis,
  34 (2003), pp.~311--332.

\bibitem{bourgeat1994stochastic}
{\sc A.~Bourgeat, A.~Mikelic, and S.~Wright}, {\em Stochastic two-scale
  convergence in the mean and applications.}, Journal f{\"u}r die reine und
  angewandte Mathematik, 456 (1994), pp.~19--52.

\bibitem{cohn2013measure}
{\sc D.~L. Cohn}, {\em Measure theory}, vol.~5, Springer, 2013.

\bibitem{collin2018mathematical}
{\sc A.~Collin and S.~Imperiale}, {\em Mathematical analysis and 2-scale
  convergence of a heterogeneous microscopic bidomain model}, Mathematical
  Models and Methods in Applied Sciences, 28 (2018), pp.~979--1035.

\bibitem{cornfeld2012ergodic}
{\sc I.~P. Cornfeld, S.~V. Fomin, and Y.~G. Sinai}, {\em Ergodic theory},
  vol.~245, Springer Science \& Business Media, 2012.

\bibitem{fitzhugh1955mathematical}
{\sc R.~FitzHugh}, {\em Mathematical models of threshold phenomena in the nerve
  membrane}, The Bulletin of Mathematical Biophysics, 17 (1955), pp.~257--278.

\bibitem{franzone2002degenerate}
{\sc P.~C. Franzone and G.~Savar{\'e}}, {\em Degenerate evolution systems
  modeling the cardiac electric field at micro-and macroscopic level}, in
  Evolution Equations, Semigroups and Functional Analysis, Springer, 2002,
  pp.~49--78.

\bibitem{GraKar2019}
{\sc E.~Grandelius and K.~H. Karlsen}, {\em The cardiac bidomain model and
  homogenization}, Networks \& Heterogeneous Media, 14 (2019), p.~173.

\bibitem{heida2022stochastic-II}
{\sc M.~Heida}, {\em Stochastic homogenization on perforated domains
  ii--application to nonlinear elasticity models}, ZAMM-Journal of Applied
  Mathematics and Mechanics/Zeitschrift f{\"u}r Angewandte Mathematik und
  Mechanik, 102 (2022), p.~e202100407.

\bibitem{hille1996functional}
{\sc E.~Hille and R.~S. Phillips}, {\em Functional analysis and semi-groups},
  vol.~31, American Mathematical Soc., 1996.

\bibitem{hodgkin1952}
{\sc A.~L. Hodgkin and A.~F. Huxley}, {\em A quantitative description of
  membrane current and its application to conduction and excitation in nerve},
  The Journal of physiology, 117 (1952), p.~500.

\bibitem{ivorra2010electrical}
{\sc A.~Ivorra, J.~Villemejane, and L.~M. Mir}, {\em Electrical modeling of the
  influence of medium conductivity on electroporation}, Physical Chemistry
  Chemical Physics, 12 (2010), pp.~10055--10064.

\bibitem{jankowiak2020comparison}
{\sc G.~Jankowiak, C.~Taing, C.~Poignard, and A.~Collin}, {\em Comparison and
  calibration of different electroporation models. application to rabbit livers
  experiments}, ESAIM: Proceedings and Surveys, 67 (2020), pp.~242--260.

\bibitem{JEREZHANCKES2023103789}
{\sc C.~Jerez-Hanckes, I.~A. {Martínez Ávila}, I.~Pettersson, and
  V.~Rybalko}, {\em Derivation of a bidomain model for bundles of myelinated
  axons}, Nonlinear Analysis: Real World Applications, 70 (2023), p.~Paper No.
  103789.

\bibitem{jerez2020derivation}
{\sc C.~Jerez-Hanckes, I.~Pettersson, and V.~Rybalko}, {\em Derivation of cable
  equation by multiscale analysis for a model of myelinated axons}, Discrete
  and Continuous Dynamical Systems-Series B, 25 (2020), pp.~815--839.

\bibitem{jikov2012homogenization}
{\sc V.~V. Jikov, S.~M. Kozlov, and O.~A. Oleinik}, {\em Homogenization of
  differential operators and integral functionals}, Springer Science \&
  Business Media, 2012.

\bibitem{lions1969quelques}
{\sc J.-L. Lions}, {\em Quelques m{\'e}thodes de r{\'e}solution de problemes
  aux limites non lin{\'e}aires}, Dunod, 1969.

\bibitem{Majda_Bertozzi_2001}
{\sc A.~J. Majda and A.~L. Bertozzi}, {\em Vorticity and Incompressible Flow},
  Cambridge Texts in Applied Mathematics, Cambridge University Press, 2001.

\bibitem{mandonnet2011role}
{\sc E.~Mandonnet and O.~Pantz}, {\em The role of electrode direction during
  axonal bipolar electrical stimulation: a bidomain computational model study},
  Acta Neurochirurgica, 153 (2011), pp.~2351--2355.

\bibitem{mecke1967stationare}
{\sc J.~Mecke}, {\em Station{\"a}re zuf{\"a}llige ma{\ss}e auf lokalkompakten
  abelschen gruppen}, Zeitschrift f{\"u}r Wahrscheinlichkeitstheorie und
  verwandte Gebiete, 9 (1967), pp.~36--58.

\bibitem{mikelic2023homogenization}
{\sc A.~Mikeli{\'c} and A.~Piatnitski}, {\em Homogenization of the linearized
  ionic transport equations in random porous media}, Nonlinearity, 36 (2023),
  p.~3835.

\bibitem{minty1962}
{\sc G.~Minty}, {\em Monotone (nonlinear) operators in hilbert space}, Duke
  Math. J., 29 (1962), pp.~341--346.

\bibitem{miyazawa1995note}
{\sc M.~Miyazawa}, {\em Note on generalizations of mecke's formula and
  extensions of h= $\lambda$g}, Journal of applied probability, 32 (1995),
  pp.~105--122.

\bibitem{nagumo1962active}
{\sc J.~Nagumo, S.~Arimoto, and S.~Yoshizawa}, {\em An active pulse
  transmission line simulating nerve axon}, Proceedings of the IRE, 50 (1962),
  pp.~2061--2070.

\bibitem{neu1993homogenization}
{\sc J.~C. Neu and W.~Krassowska}, {\em Homogenization of syncytial tissues.},
  Critical reviews in biomedical engineering, 21 (1993), pp.~137--199.

\bibitem{pennacchio2005}
{\sc M.~Pennacchio, G.~Savar{\'e}, and P.~C. Franzone}, {\em Multiscale
  modeling for the bioelectric activity of the heart}, SIAM Journal on
  Mathematical Analysis, 37 (2005), pp.~1333--1370.

\bibitem{pettersson2023bidomain}
{\sc I.~Pettersson, A.~Rybalko, and V.~Rybalko}, {\em Bidomain model for axon
  bundles with random geometry}, in Gas Dynamics with Applications in Industry
  and Life Sciences: On Gas Kinetic/Dynamics and Life Science Seminar, March
  25--26, 2021 and March 17--18, 2022, vol.~429, Springer Nature, 2023, p.~93.

\bibitem{piatnitski2020homogenization}
{\sc A.~Piatnitski and M.~Ptashnyk}, {\em Homogenization of biomechanical
  models of plant tissues with randomly distributed cells}, Nonlinearity, 33
  (2020), p.~5510.

\bibitem{richardson2011}
{\sc G.~Richardson and S.~J. Chapman}, {\em Derivation of the bidomain
  equations for a beating heart with a general microstructure}, SIAM Journal on
  Applied Mathematics, 71 (2011), pp.~657--675.

\bibitem{rioux2013predictive}
{\sc M.~Rioux and Y.~Bourgault}, {\em A predictive method allowing the use of a
  single ionic model innumerical cardiac electrophysiology}, ESAIM:
  Mathematical Modelling and Numerical Analysis, 47 (2013), pp.~987--1016.

\bibitem{shiryaev2016probability}
{\sc A.~N. Shiryaev}, {\em Probability-1}, vol.~95, Springer, 2016.

\bibitem{showalter2013monotone}
{\sc R.~E. Showalter}, {\em Monotone operators in Banach space and nonlinear
  partial differential equations}, vol.~49, American Mathematical Soc., 2013.

\bibitem{standring2021gray}
{\sc S.~Standring}, {\em Gray's anatomy e-book: the anatomical basis of
  clinical practice}, Elsevier Health Sciences, 2021.

\bibitem{walters2007ergodic}
{\sc P.~Walters}, {\em Ergodic theory—introductory lectures}, vol.~458,
  Springer, 2007.

\bibitem{zhikov1993averaging}
{\sc V.~V. Zhikov}, {\em Averaging in perforated random domains of general
  type}, Mathematical Notes, 53 (1993), pp.~30--42.

\end{thebibliography}

\end{document}